\documentclass[a4paper,10pt]{article}

\usepackage{amsmath}
\usepackage{amssymb}
\usepackage{amsthm}
\usepackage{amsfonts}
\usepackage{xcolor}
\usepackage{graphicx}
\usepackage{mathtools}
\usepackage{enumerate}
\usepackage{multirow}
\usepackage{rotating}
\usepackage{setspace}
\usepackage{lineno}
\usepackage{url}

\theoremstyle{theorem}
\newtheorem{thm}{Theorem}

\newtheorem{lem}{Lemma}

\theoremstyle{definition}

\usepackage{geometry}

\begin{document}
	
	\title{ \large{Time Optimal Control Studies on COVID-19 Incorporating  Adverse Events of the Antiviral Drugs }}
	
	\vspace{0.1in}
	\author{{\small Bishal Chhetri$^{1,a}$, Vijay M. Bhagat$^{b}$, Swapna Muthusamy$^{b}$, Ananth V S$^{a}$,D. K. K. Vamsi$^{*,a}$,  Carani B Sanjeevi$^{c}$}\hspace{2mm} \\
		{\it\small $^{a}$Department of Mathematics and Computer Science, Sri Sathya Sai Institute of Higher Learning, Prasanthi Nilayam}, \\
		{\it \small Puttaparthi, Anantapur District - 515134, Andhra Pradesh, India}\\
		{\it\small $^{b}$Central Leprosy Teaching and Research Institute - CLTRI, Chennai, India}\\
		{\it\small $^{c}$  Vice-Chancellor, Sri Sathya Sai Institute of Higher Learning -  SSSIHL, India}\\
		{\it\small bishalchhetri@sssihl.edu.in, dkkvamsi@sssihl.edu.in$^{*}$,}\\
		{\it\small sanjeevi.carani@sssihl.edu.in, sanjeevi.carani@ki.se}\\
		{\small $^{1}$  First Author},
		{ \small $^{*}$ Corresponding Author}
		\vspace{1mm}
	}
	
	\date{}
	\maketitle

\begin{abstract} \vspace{.25cm}

The COVID -19 pandemic has resulted in more than 166 million infections and 3.4 million deaths worldwide. Several drug interventions targeting multiple stages of the pathogenesis of COVID -19 can significantly reduce induced infection and thus mortality. In this study, we first develop SIV model by incorporating the intercellular time delay and analyze the stability of the equilibrium points. The model dynamics admits disease-free equilibrium and the infected equilibrium with their stability, based on the value of basic reproduction number $R_0$. We then frame an optimal control problem with antiviral drugs and second-line drugs as control measures and study their roles in reducing the infected cell count and the viral load. The comparative study done in the optimal control problem suggests that when the first line antiviral drugs shows adverse events, considering these  drugs in reduced quantity along with the second line drug would be highly effective in reducing the infected cell and viral load in a COVID infected patients. Later, we formulate a time-optimal control problem with the objective to drive the system from any given initial state to the desired infection-free equilibrium state in minimal time. Using Pontryagin's Minimum Principle the optimal control strategy is shown to be of bang-bang type with possibility of switches between two extreme values of the optimal controls. Numerically, it is shown that the desired infection-free state is achieved in less time when the higher values of both the optimal controls are chosen. The results obtained from this study can be very helpful to researchers, epidemiologists, clinicians, and doctors who are working in this field. 

\end{abstract}
\section*{keywords} \vspace{.25cm}
 COVID-19; First Line Drugs; Second Line Drugs; Adverse Events; Time Optimal Control Problem; Bang-Bang Control;

\section{Introduction}

 Mathematical modeling of infectious
diseases is one of the most important researched area today. Mathematical epidemiology has contributed to a better
understanding of the dynamical behavior of  infectious diseases,
its impacts, and possible future predictions about its
spreading. Mathematical models are used in comparing,
planning, implementing, evaluating, and optimizing various
detection, prevention, therapy, and control programs. COVID-19 is one such contagious respiratory and vascular disease that has shaken the world today. It is  caused by severe acute respiratory syndrome coronavirus 2 (SARS-CoV-2). On 30 january it was declared as a public health emergency of international concern \cite{concern}. COVID-19 has resulted in around 166 million cases and  3.4 million deaths worldwide. Several mathematical models has been developed  to understand the dynamics  of the disease. Various compartment models to understand the dynamical behavior of COVID-19 can be found in \cite{samui2020mathematical, ndairou2020mathematical, zeb2020mathematical,leontitsis2021seahir, wang2020four, dashtbali2021compartmental, zhao2020five, chen2020mathematical, biswas2020covid, sarkar2020modeling}.  The optimal control studies to examine the role of control policies such  vaccination, treatment, quarantine, isolation, and screening for controlling COVID-19 infection can be found in  \cite{kkdjou2020optimal,libotte2020determination,aronna2020model, araz2021analysis,ndondo2021analysis, dhaiban2021optimal,gollmann2009optimal}.

The time-optimal control problem in SIR (Susceptible-Infected-Recovered) epidemic models is discussed, 
with different control policies such as vaccination, isolation, culling, and reduction of transmission in \cite{bolzoni2017time}. For all the policies investigated, the optimal control is shown to be of  bang–bang type. The results of this study suggests that  when a switch occurs between the optimal control values, the optimal strategy is to delay the
control action some amount of time and then apply the control at the maximum rate for the remainder
of the outbreak. A concept of the optimization of control measures for SARS epidemics
spread, based on the subsystem of the compartmental model is studied in \cite{jiang2007optimal}. Using Pontryagin’s minimum principle it is proved  that a
maximum quarantine/isolation measures  would reduce SARS epidemics to minimum extent in minimum time.  
The treatment of Covid-19 disease can be mainly classified in two settings. The initial period consisting of viral multiplication and body’s efforts to contain the spread of virus. In this phase there may not be any systematic symptoms such as breathlessness or the need of hospitalization and oxygen support. This initial period of the disease the body requires supports to fight against the infection, and therefore, mostly requires symptomatic treatment and supportive management. As the viral multiplication is one of the initial concerns, an antiviral drug such as the Remdesivir is the mainstay of the treatment \cite{mm}. 

In the later half in the view to contain and eliminate the virus the exaggerated immune response resulting mostly the compromised body’s essential functions such as increased respiratory rate to maintain the function of oxygenation etc. In this stage there require hospitalization and oxygen support. At this stage it is important to suppress the exaggerated immune response which is hazardous to maintain the essential functions. Therefore, the corticosteroids have important role to play here. Therefore, it is recommended to use dexamethasone \cite{recovery2021dexamethasone}, which has proven to improve the clinical outcome among patients who are in the later phase of disease and require oxygen support either through noninvasive ventilation or ECMO. If dexamethasone is not available the prednisolone or methyl prednisolone is also recommended with or without the combination of antiviral agents as there is immunosuppression due to steroids \cite{mm}. Therefore, the antivirals may be considered as first line drugs and the corticosteroids and the anti-inflammatory drugs as second line therapeutic modality.\\

Patients affected by moderate to severe COVID-19 pneumonia, who failed to respond to azithromycin, hydroxychloroquine and two doses of TCZ were evaluated in \cite{conticini2020high}. In all five patients, hydroxychloroquine and azithromycin were immediately administered at diagnosis, whereas intravenous TCZ, 8 mg/kg, within 72 hours from hospitalisation, and then repeated
after 24 hours. None of them reported substantial benefit after anti-IL-6 treatment and one patient required ICU admission and IV. From 3 to 5 days after the first administration of TCZ, all subjects were treated with intravenous methylprednisolone (MP) 1.5 mg/kg, slowly tapered after 5 days. It was observed that all the five patients evidenced a prompt and remarkable improvement: within 7 days, all three subjects in ICU
did not require IV anymore and were awakened. \cite{conticini2020high} confirmed the
evidence  about a possible synergic role of TCZ and methylprednisolone (MP) in limiting the exaggerating autoimmune response leading to ARDS. 

From the above clinical studies, it is clear that when patients fail to respond to the antiviral drugs such as azithromycin and hydroxychloroquine, considering the second line drugs such as methylprednisolone, TCZ and methylprednisolone (MP) after few days would be highly effective in improving the condition of the patients. In this context, within-host mathematical modelling  can be extremely helpful in understanding the efficacy of these interventions. Modelling the within-host dynamics of Covid-19 incorporating the adverse events of the antiviral drugs and studying the time optimal control problem, which is being attempted here is the first of its kind for COVID-19.  

In the article \cite{chhetri2020optimal}, the authors have done an extensive study on the role of antiviral drugs such as Arbidol, Remdesivir, Lopinavir and Ritonavir and immunomodulators such as INF, and Zinc in COVID-19 infection. In this work initially, we extend the work done in \cite{chhetri2020within} by incorporating inter-cellular time delay and do the stability analysis of the equilibrium points admitted by the model. Secondly, an optimal control problem is framed with antiviral agents and second-line drugs as control measures incorporating the adverse events caused by antiviral drugs and their roles in reducing the infected cell count and the viral load is studied. Lastly, a time-optimal control problem is formulated with the objective to drive the system from any given initial state to the desired infection-free equilibrium state in minimal time.

\section{Model Without Interventions}

Many within-host mathematical models have been developed to understand the dynamics of infectious diseases such as HIV, dengue influenza, and COVID -19 \cite{perelson1999mathematical, mishra2017micro, pawelek2019correction, vargas2020host}. Most of these studies ignore the intercellular delay by assuming that the infectious process is instantaneous, which may not be biologically true \cite{hattaf2012optimal}. 
A detailed study of the SIV model developed on the basis of pathogenesis dealing with COVID -19 is carried out by the authors in \cite{chhetri2020within}. In the present work, a model is developed and studied that takes into account the intercellular time delay. This model is described by the following system of differential equations.

\begin{eqnarray}
   	\frac{dS}{dt}& =&  \omega \ - \beta SV  - \mu S  \label{sec2equ1} \\
   	\frac{dI}{dt} &=& \beta S(t-\tau)V(t-\tau) \ -  { \bigg(d_{1}  + d_{2}  + d_{3} +  d_{4} + d_{5}+ d_{6}\bigg)I }  \ - \mu I   \label{sec2equ2}\\ 
   	\frac{dV}{dt} &=&  \alpha I   \ -  \bigg( b_{1}  + b_{2}  + b_{3}  +  b_{4}+ b_{5} + b_{6}\bigg)V    \ -  \mu_{1} V \label{sec2equ3}
   \end{eqnarray} 
   
The meaning of each of the variables and parameters of the model is given in table 1.   \\

   \begin{table}[ht!]
     	\caption{Meanings of the Variables and Parameters}
     	\centering 
     	\begin{tabular}{|l|l|} 
     		\hline\hline
     		
     		\textbf{Parameters/Variables} &  \textbf{Biological Meaning} \\  
     		\hline\hline 
     		$S$ & Healthy Type II Pneumocytes  \\
     		\hline\hline
     		
     		$I$ & Infected Type II Pneumocytes  \\
     		\hline\hline
     		$\omega$ & Natural birth rate of Type II Pneumocytes \\
     		\hline\hline
     		$V$ & Viral load  \\
     		\hline\hline
     		$\beta$ & Infection rate  \\
     			\hline\hline
     		$b$ & Burst rate \\
     		
     		\hline\hline
     		$\mu$ & Natural death rate of Type II Pneumocytes \\
     		\hline\hline
     		$\mu_{1}$ & Natural death rate of virus \\
     		\hline\hline
     		
     		$d_{1}, \hspace{.25cm} d_{2}, \hspace{.25cm} d_{3}, \hspace{.25cm} d_{4}, \hspace{.25cm} d_{5}, \hspace{.25cm} d_{6}$ & Rates at which virus is removed because\\
     	& the release of cytokines and chemokines  IL-6\\
     	&  TNF-$\alpha$, \hspace{.2cm}INF-$\alpha$,  \hspace{.2cm}CCL5, \hspace{.2cm}CXCL8 , \hspace{.2cm}CXCL10   \hspace{.2cm} respectively   \\ 
     	\hline\hline
     			
     	$b_{1}, \hspace{.25cm} b_{2}, \hspace{.25cm} b_{3}, \hspace{.25cm} b_{4}, \hspace{.25cm} b_{5}, \hspace{.25cm} b_{6}$ & Rates at which infected cell is removed because of\\
     		& the release of cytokines and chemokines  IL-6\\
     		&  TNF-$\alpha$, \hspace{.2cm}INF-$\alpha$,  \hspace{.2cm}CCL5, \hspace{.2cm}CXCL8 , \hspace{.2cm}CXCL10   \hspace{.2cm} respectively   \\ 
     		\hline\hline
     		$\tau$ & Inter-cellular delay \\
     		\hline\hline
     	\end{tabular}
     \end{table} \vspace{.25cm}
 
 \newpage
 \subsection{\textbf{Positivity and Boundednes}}
 
 In this subsection we will show that the system $(2.1)-(2.3)$ remains positive and bounded for all time $t$.
 Let
 $$ x=\bigg(b_{1}  + b_{2}  + b_{3} +  b_{4} + b_{5}+ b_{6}\bigg)$$
 $$y=\bigg(d_{1}  + d_{2}  + d_{3} +  d_{4} + d_{5}+ d_{6}\bigg)$$
 
 \underline{\textbf{Positivity}}\textbf{:}
 We now show that if the initial conditions of the system $(2.1)-(2.3)$ are positive, then the solution remain positive for any future time. Using the  equations $(2.1)-(2.3)$  we get,
\begin{align*}
\frac{dS}{dt} \bigg|_{S=0} &= \omega \geq 0 ,  &  
\frac{dI}{dt} \bigg|_{I=0} &= \beta S(t-\tau) V(t-\tau)  \geq 0,
\\ \\
\frac{dV}{dt} \bigg|_{V=0} &= b I \geq 0.  &
\end{align*}

\vspace{1.5mm}
\noindent
\\ 
Thus all the above rates are non-negative on the bounding planes (given by $S=0$, $I=0$, and $V=0$) of the non-negative region of the real space. So, if a solution begins in the interior of this region, it will remain inside it throughout time $t$. This  happens because the direction of the vector field is always in the inward direction on the bounding planes as indicated by the above inequalities. Hence, we conclude that all the solutions of the the system $(2.1)-(2.3)$ remain positive for any time $t>0$  provided that the initial conditions are positive. This establishes the positivity of the solutions of the system $(2.1)-(2.3)$.

\underline{\textbf{Boundedness}}\textbf{:}\\
Let  $N(t) = S(t)+I(t+\tau)+V(t+\tau) $ \\
Now,  
\begin{equation*}
\begin{split}
\frac{dN}{dt} & = \frac{dS}{dt} +  \frac{dI}{dt}+  \frac{dV}{dt}  \\[4pt]
& = \omega -\mu S(t) -x I(t+\tau)-\mu I(t+\tau) + b I(t+\tau)-y V(t+\tau)-\mu_{1}V(t+\tau) \\
& = \omega -(x-b) I(t+\tau)-\mu (I(t+\tau)+S(t)) -y V(t+\tau)-\mu_{1}V(t+\tau) \\
& \le \omega -k(S(t)+I(t+\tau)+V(t+\tau))\\
& = \omega - k N(t) \\
\end{split}
\end{equation*}
with the assumption that $x > \alpha$ and $k= \min(\mu_{1},\mu)$  

Here the integrating factor is $e^{-k t}.$ Therefore, after integration we get,

$N(t)\le \frac{\omega}{k} + ce^{-k t}.$ Now  as $t \rightarrow \infty$ we get, 
$$N(t)\le \frac{\omega}{k}$$
Thus we have shown that the system $(2.1)-(2.3)$ is positive and bounded. Therefore, the biologically feasible region is given by the following set, 
\begin{equation*}
\Omega = \bigg\{\bigg(S(t), I(t), B(t)\bigg) \in \mathbb{R}^{3}_{+} : S(t)  + I(t+\tau) + V(t+\tau) \leq \frac{\omega}{\mu}, \ t \geq 0 \bigg\}
\end{equation*}

We summarize the above discussion on positivity and boundedness of the system  $(2.1)-(2.3)$ by the following theorem. 

\begin{thm}
    Let $k=\min\{\mu, \mu_1\}$ and $x > \alpha$. Then the set \\
    $$\Omega = \bigg\{\bigg(S(t), I(t), V(t)\bigg) \in \mathbb{R}^{3}_{+} : S(t)  + I(t+\tau) + V(t+\tau) \leq \frac{\omega}{k}, \ t \geq 0 \bigg\}$$
    is a positive invariant and an attracting set for system $(2.1)-(2.3)$.
\end{thm}

 \subsection{{\textbf{Equilibrium Points and Basic Reproduction Number ($R_{0}$}) }}\vspace{.25cm}	

System $(2.1)-(2.3)$ admits two equilibria  namely, the infection free equilibrium $E_{0}=\bigg(\frac{\omega}{\mu},0,0 \bigg)$ and the infected equilibrium $E_{1}=(S^*, I^*, V^*)$  where,

$$\hspace{-1cm}S^*=\frac{(y+\mu_{1})(x+\mu)}{b \beta}$$

$$I^*=\frac{\alpha \beta \omega-\mu (y+\mu_{1})(x+\mu)}{b \beta (x+\mu)}$$

$$V^*=\frac{b \beta \omega-\mu(y+\mu_{1})(x+\mu)}{ \beta (x+\mu)(y+\mu_{1})}$$

Now  we calculate the basic reproduction number which is the most important quantity in any infectious disease models.  The basic reproduction number is calculated using the next generation matrix method \cite{diekmann2010construction} and the expression for $R_{0}$ for the system $(2.1)-(2.3)$ is given by \vspace{.5cm}\\

\begin{equation}
\mathbf{ R_{0}}= \mathbf{\frac{\beta b \omega}{\mu (x+\mu) (y+\mu_{1})}} \label{sec3equ1}\\
\end{equation}

With the definition of $R_0$, the infected equilibrium $E_1$ can be re-written as,
$E_{1}=(S^*, I^*, V^*)$  where,
$$\hspace{-1cm}S^*=\frac{\omega}{R_0 \mu}$$

$$I^*=\frac{\omega (R_0 -1)}{R_0 (x+\mu)}$$

$$V^*=\frac{\mu (R_0-1)}{\beta}$$
Since negative population does not make sense, the existence condition for the infected equilibrium point $E_1$ is that $R_0 > 1$.

\subsection{\textbf{Stability Analysis}}
In this section we analyse the stability of  equilibrium points $E_0$ and $E_1$ admitted by the system $(2.1)-(2.3)$.  This is done based on the nature of the eigen values of the jacobian matrix evaluated at each of the equilibrium point.

\subsubsection{\textbf{Stability of $E_0$}}
The jacobian matrix of the system $(2.1)-(2.3)$ at the infection free equilibrium $E_0$ is given by, \\

\begin{equation*}
J_{E_{0}} = 
\begin{pmatrix}
-\mu & 0 & \frac{-\beta\omega}{\mu} \\
0 & -(x+\mu) & \frac{\beta\omega e^{-\lambda \tau}}{\mu} \\
0 & \alpha & -(y+\mu_{1})
\end{pmatrix}
\end{equation*} \\
 
 The characterstic equation of $J_{E_0}$ is given by,
 \begin{equation}
 (-\mu-\lambda)\bigg(\lambda^2+(x+y+\mu+\mu_1)\lambda + (x+\mu)(y+\mu_1)-\frac{\beta \omega b e^{-\lambda \tau}}{\mu}\bigg) \label{111b}
 \end{equation}
 One of the eigenvalue of characteristic equation $(\ref{111b})$ is $-\mu$ and the other two are the roots of the following equation.\\
 \begin{equation}
 \bigg(\lambda^2+(x+y+\mu+\mu_1)\lambda + (x+\mu)(y+\mu_1)-\frac{\beta \omega b e^{-\lambda \tau}}{\mu}\bigg) \label{e0}
 \end{equation}
 when $\tau=0$, substituting for $R_0$ in $(2.6)$ we get the characterstic equation of the form,
  \begin{equation}
 \bigg(\lambda^2+(x+y+\mu+\mu_1)\lambda - (R_0-1)(x+\mu)(y+\mu_1)\bigg)
 \end{equation}

We see that when $R_0 < 1$ all the eigenvalues of equation $(2.7)$ is negative. Therefore, $E_0$ remains locally asymptotically stable for $\tau=0$ whenever $R_0 < 1$.

Now we will examines the nature of the eigenvalues of equation $(\ref{e0})$ for the case $\tau \neq 0$. For examining the stability of $E_0$ with delay we will assume that $E_0$ is asymptotically stable for the case with $\tau=0$.

Let $\lambda \;=\; \mu(\tau) + i\omega(\tau)$ where, $\mu$ and $\omega$ are real. Since $E_0$ is asymptotically stable for $\tau=0$,  $\mu(0) < 0$. We will choose $\tau$ sufficiently close to 0 and use continuity of $\tau$ to examine the stability of $E_0$ for $\tau \neq 0$. Let $\tau > 0$ be sufficiently small, then by continuity $ \mu(\tau) < 0 $ and $E_0$ will still remain stable. The stability changes for some values of $\tau$ for which $\mu(\tau) \;=\; 0 $ and $\omega(\tau)\ne 0 $ that is when $\lambda$ is purely imaginary. Let $\tau^*$ be such that $\mu(\tau^*) \;=\;0 $ and $\omega(\tau^*)\ne 0$. In this case the steady state loses stability and becomes unstable when $\mu(\tau)$ becomes positive.\\
By Rouche’s theorem \cite{dieudonne2011foundations} the transcendental equation $(\ref{e0})$ has roots with positive real parts if
and only if it has purely imaginary roots. Now we will assume that the characteristic equation $(\ref{e0})$
has purely imaginary roots and then arrive at a contradiction.

Let $$\lambda = i\omega$$ where $\omega > 0$ is real.\\
 Let $$s=x+y+\mu+\mu_1$$ 
 $$m= (x+\mu)(y+\mu)$$ 
 
  The expression for $R_0$ defined in $(2.4)$ in terms of $m$ is given by,
 $$R_0= \frac{\beta b \omega}{m \mu}$$
 Substituting all these in $(\ref{e0})$ we get\\
 $$\lambda^2 + s \lambda + m(1-R_0 exp^{-\lambda \tau})=0$$ 
 substituting $\lambda = i \omega$ we get 
 $$m-\omega^2+is\omega = m R_0(cos\omega\tau - i sin\omega\tau)$$
 Comparing the real and the imaginary part we get,\\
 $$m-\omega^2=m R_0 cos\omega\tau$$
 $$s\omega=- m R_0sin\omega\tau$$
 Adding and squaring the above equations we get,\\
 $$(m-\omega^2)^2+ s^2\omega^2= m^2 R_0^2$$
 Simplifying we get,\\
 \begin{equation}
    \omega^4+(s^2-2m)\omega^2+m^2(1-R_0^2) \label{w}
 \end{equation}
 
 From the definition of $s$ and $m$ we see that $(s^2 - 2m) > 0$. 
 
 Now since $R_0 < 1$,  all roots of $(\ref{w})$ are imaginary. This contradicts the fact that $\omega$ is real. Therefore, from Rouche's theorem we conclude that transcendental equation  $(\ref{e0})$ has all roots with negative real part. Therefore, $E_0$ remains asymptotically stable for all values of delay whenever $R_0 < 1$. 
 
 When $R_0=1$ then it is clear that $\omega =0 $ is a simple root of  $(\ref{w})$. This also leads to a contradiction since $\omega$ was assumed to be strictly positive. Therefore, with $R_0=1$ all the roots of $(\ref{e0})$ has negative real parts except $\lambda=0$. This implies that the infection free equilibrium $E_0$ is asymptotically stable.
 

Therefore, it follows from the continuity of $f(\lambda)$ on $(-\infty, + \infty)$ that equation $f(\lambda)=0$ has at least one positive root. Hence the characteristic equation $(\ref{e0})$  has atleast one positive root. Hence, the infection free equilibrium $E_0$ is unstable for $R_0 > 1$. All the discussion above is summarised by the following theorem. \\

\begin{thm}
      The infection free equilibrium point $E_0$ of  system $(2.1)-(2.3)$ is locally asymptotically stable for any time delay $\tau$ provided $R_0 < 1$. If $R_0$ crosses unity $E_0$ loses its stability and becomes unstable.   
\end{thm}
 
 \subsubsection{\textbf{Stability of $E_1$}}

 The jacobian matrix of the system $(2.1)-(2.3)$ at $E_1$ is given by,

\begin{equation*}
J = 
\begin{pmatrix}
-\beta V^*-\mu & 0 & -\beta S^* \\
\beta V^{*} & -(x+\mu) & \beta S^{*} \\
0 & b & -(y+\mu_{1})
\end{pmatrix}
\end{equation*} 

The characteristic equation of the jacobian $J$ evaluated at $E_{1}$ is given by, \\
\begin{equation}
\lambda^3+(s+C)\lambda^2+(sC+(m-Ee^{-\lambda\tau}))\lambda+De^{-\lambda\tau}+(m-Ee^{-\lambda\tau})C=0
 \label{l}  \\
\end{equation}
 where $$C=\beta V^*+\mu$$
$$E=\beta b S^*$$
$$D=\beta^2 b S^* V^*$$
 when $\tau=0$ the characteristic equation  $(\ref{l})$ with the definition of $R_0$ reduces to\\
 \begin{equation}
\lambda^3 + \bigg(p+\mu R_{0}\bigg)\lambda^2 + \bigg(p\mu R_{0}\bigg)\lambda + q\mu \bigg(R_{0}-1\bigg) = 0 \label{sec3equ4}  \\
\end{equation}
where $p=x+y+\mu+\mu_{1}$ and $q=(x+\mu)(\mu_{1}+y).$ \\
Therefore, we see that whenever $R_0 > 1$,  all the roots of equation $(2.10)$ is negative implying the asymptotic stability of $E_1$. To study the stability of $E_1$ with for $\tau \neq 0$, we substitute $\lambda=i\omega$ in the characteristic equation $(\ref{l})$  and arrive at a contradiction in similar lines the stability of $E_0$ discussed earlier.

Substituting $\lambda=i \omega$ in equation  $(\ref{l})$ we get,\\
$$-(s+C)\omega^2-\omega E sin(\omega\tau)-(\omega^3-(sC+m)\omega+\omega E cos(\omega \tau))i=-mC-(D-EC)cso(\omega\tau)+(D-EC)sin(\omega\tau)i$$
  Comparing real and imaginary part we get
$$-\omega^3+(sC+m)\omega=\omega E cos(\omega\tau)+(D-EC)sin(\omega\tau)$$
 $$mC-(s+C)\omega^2=\omega E sin(\omega\tau)-(D-EC)cos(\omega\tau)$$
 squaring and adding we get,\\
 \begin{equation}
 \omega^6+A\omega^4+B\omega^2+(m^2C^2-(D-EC)^2=0 \label{w1}
 \end{equation}
 where $$A=(s+C)^2-2(sC+m)$$
 $$B=(sC+m)^2-2mC(s+C)-E^2$$
 
 If $A > 0, B > 0$ and $m^2C^2-(D-EC)^2 > 0$ then all the roots of $(\ref{w1})$ are all imaginary. This leads to a contradiction. Therefore, by Rouche's theorem all the eigenvalues of the characteristic equation $(\ref{l})$ have negative real parts. Hence,  $E_1$ is locally asymptotically stable for all the values of $\tau$. Substituting $m, C, s$ and $D$ in the expressions of $A$, $B$ and $m^2C^2-(D-EC)^2$ we see that $A > 0$, $B > 0$ and $m^2C^2-(D-EC)^2$ provided $R_0 >1$. Therefore, infected equilibrium point $E_1$ is asymptotically stable for all the values of $\tau$ provided $R_0 >1$. We summarize the discussion on the stability of $E_1$ by the following theorem.
 
  \begin{thm}
    The infected equilibrium point $E_1$ of  system $(2.1)-(2.3)$ is locally asymptotically stable for all the values of $\tau$ provided $R_0 >1$.  
\end{thm}

 \subsection{\textbf{Numerical Illustrations of the stability of equilibrium points}}
 
In this section we numerically illustrate the stability of the equilibrium points admitted by the system $(2.1)-(2.3)$. The simulation is done using matlab software and ode solver ode23 is used to solve the system of equation. All the parameter values used for the simulation are taken from \cite{chhetri2020within}. For the parameter values from table 2, the values of  $s^2-2m$, $m^2-P^2$ and  $R_0$ were calculated and found to be  $ 1.9336$,  $ 0.6700$ and   $0.5136$ respectively. From theorem 2.2 we know that  $E_0$ remains asymptotically stable for all the values of $\tau$ whenever $R_0 <1$, figure 1 is an illustration of theorem 2.2. The asymptotic stability of $E_0 = (20,0,0)$ for two different values of $\tau$ is depicted in figure 1. The stability of $E_0$ was checked for different values of the $\tau$ and it was found that the infection free equilibrium point remained asymptotically stable for all the values of $\tau$. \\

   \begin{table}[ht!]
	\caption{Parameter values for the stability of $E_{0}$} 
	
	\begin{center}
		\begin{tabular}{|c|c|c|c|c|c|c|c|c|c|c|c|c|c|c|c|c|}
			\hline
			$\omega$ & $\beta$ & $\mu$ & $\mu_{1}$ & $b$ & $d_{1}$ & $d_{2}$& $d_{3}$ & $d_{4}$ & $d_{5}$ & $d_{6}$ &  $b_{1}$ & $b_{2}$ & $b_{3}$ & $b_{4}$ & $b_{5}$ & $b_{6}$ \\
			\hline
			  10 & 0.05& .5 & .1 & .49 & 0.027& 0.22 &0.1 &0.428 & 0.01 & 0.01 & 0.1& .1& 0.08 & .11 & .1 & .07  \\
			\hline

		\end{tabular}
	\end{center}
\end{table}

  \begin{figure}[hbt!]
	\includegraphics[height = 6cm, width = 15.5cm]{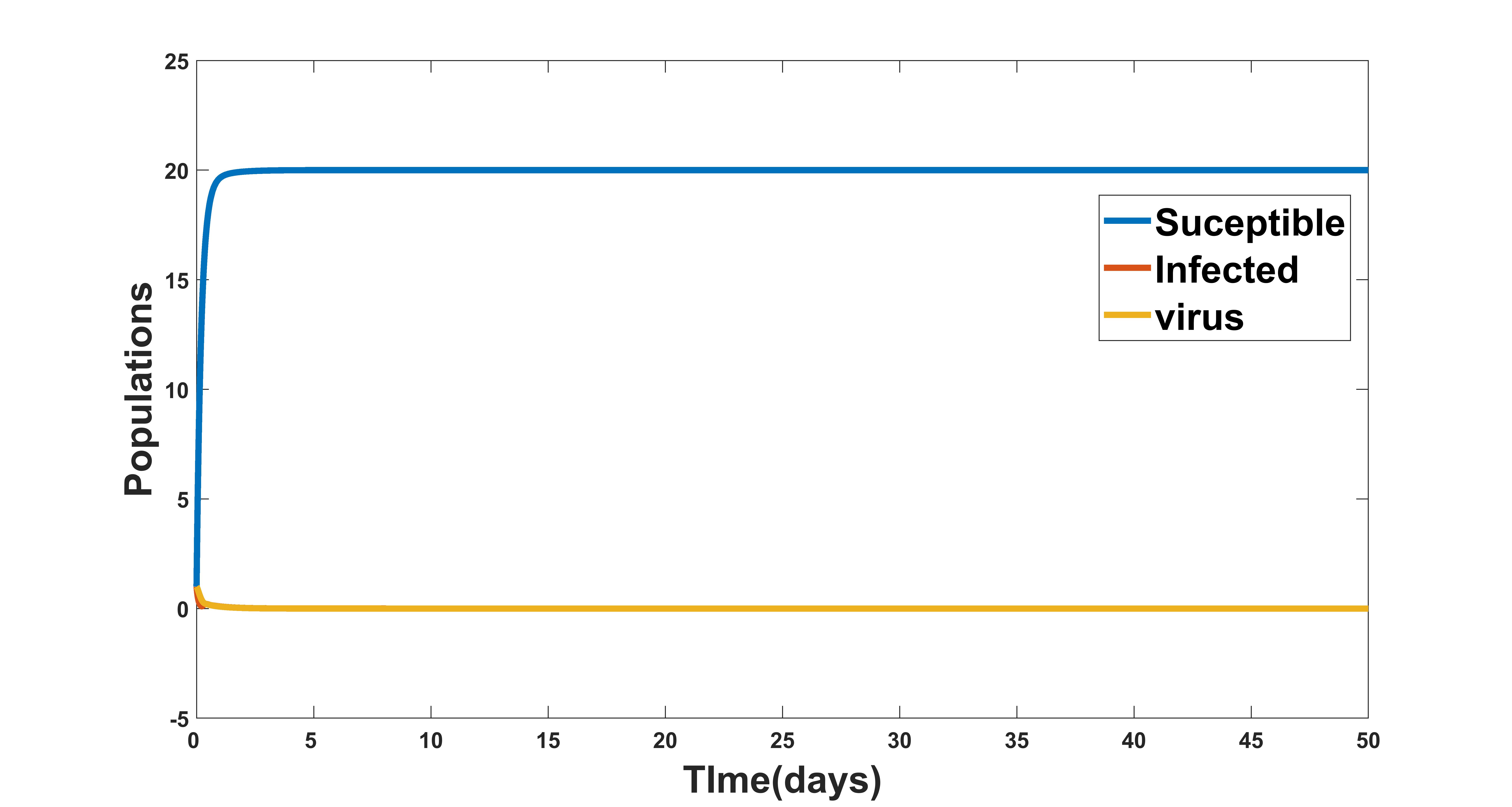}
		\includegraphics[height = 6cm, width = 15.5cm]{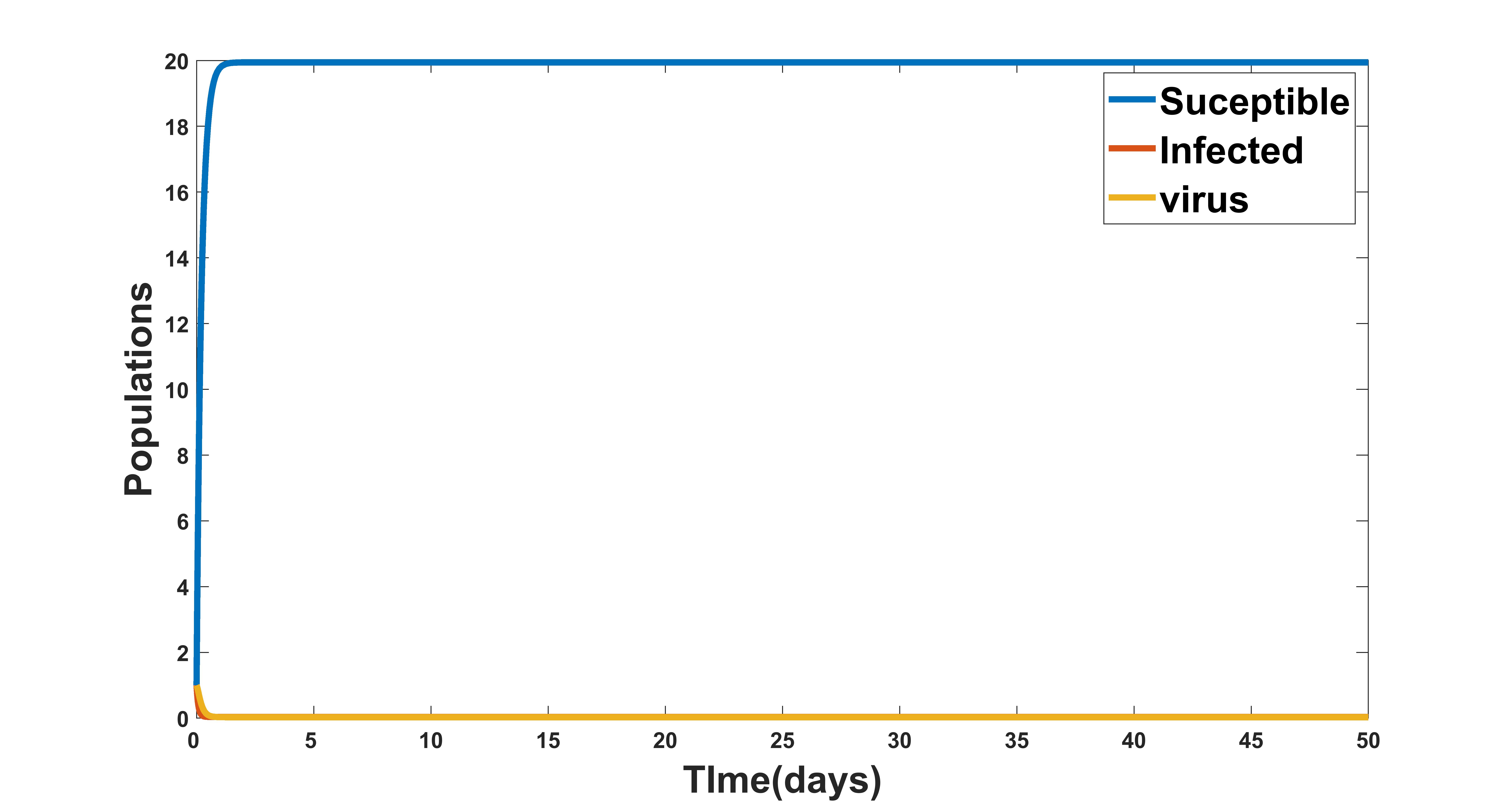}
	\caption{\label{first}Figure depicting the local asymptotic stability of $ E_{0} = (20,0,0).$ for $\tau=5$ and $\tau=15$ }
	\vspace{6mm}
\end{figure}

 \newpage
 For simulating the stability of $E_1$  the parameter values are taken from table 3. For these parameter values, the value of $R_0$ and the  infected equilibrium $E_1$ was calculated to be  $21.04 $ and $E_1=( 2.3760, 5.29, 4.008)$. Since $R_0 > 1$ for the choice of parameter values from table 3, $E_1$ remains asymptotically stable for all the values of delay ($\tau$) according to theorem 2.3. figure 2 is an illustration of theorem 2.3.  In figure 2 asymptotic stability of $E_1$ is illustrated for two different values of delay ($\tau$). Similar to the case of  stability of $E_0$,  here too the stability of $E_1$ was checked for different values of the $\tau$ and it was observed that the infected equilibrium point $E_1$ remained asymptotically stable for all the values of $\tau$,  whenever the value of basic reproduction number was above unity ($R_0 >1$).  \\ 
 
 \begin{table}[ht!]
	\caption{Parameter values for the stability of $E_{1}$} 
	
	\begin{center}
		\begin{tabular}{|c|c|c|c|c|c|c|c|c|c|c|c|c|c|c|c|c|}
			\hline
			$\omega$ & $\beta$ & $\mu$ & $\mu_{1}$ & $b$ & $d_{1}$ & $d_{2}$& $d_{3}$ & $d_{4}$ & $d_{5}$ & $d_{6}$ &  $b_{1}$ & $b_{2}$ & $b_{3}$ & $b_{4}$ & $b_{5}$ & $b_{6}$ \\
			\hline
			  5 & 0.5& .5 & .1 & .49 & 0.027& 0.22 &0.1 &0.428 & 0.01 & 0.01 & 0.1& .1& 0.08 & .11 & .1 & .07  \\
			\hline

		\end{tabular}
	\end{center}
\end{table}

 In figure 3 we plot the viral load over the period of 50 days with different values of inter-cellular delay $\tau$. The parameter values are as in table 2.  From this figure we observe that greater the value of inter-cellular delay lesser is the viral load in the body. This implies that the viral load in the body will be less if time taken by the cell  to produce new virions.

  \begin{figure}[hbt!]
	\includegraphics[height = 6cm, width = 15.5cm]{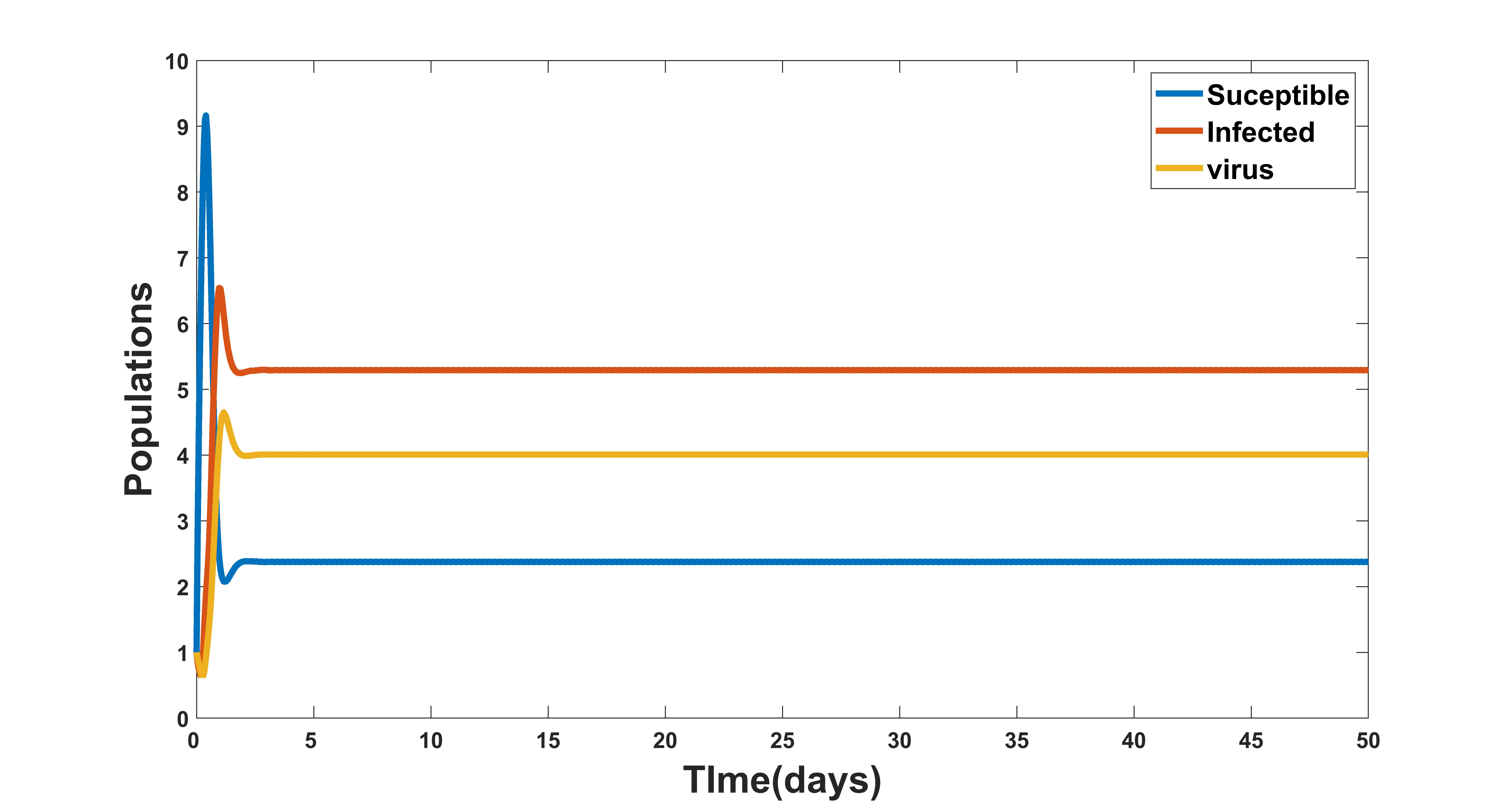}
		\includegraphics[height = 6cm, width = 15.5cm]{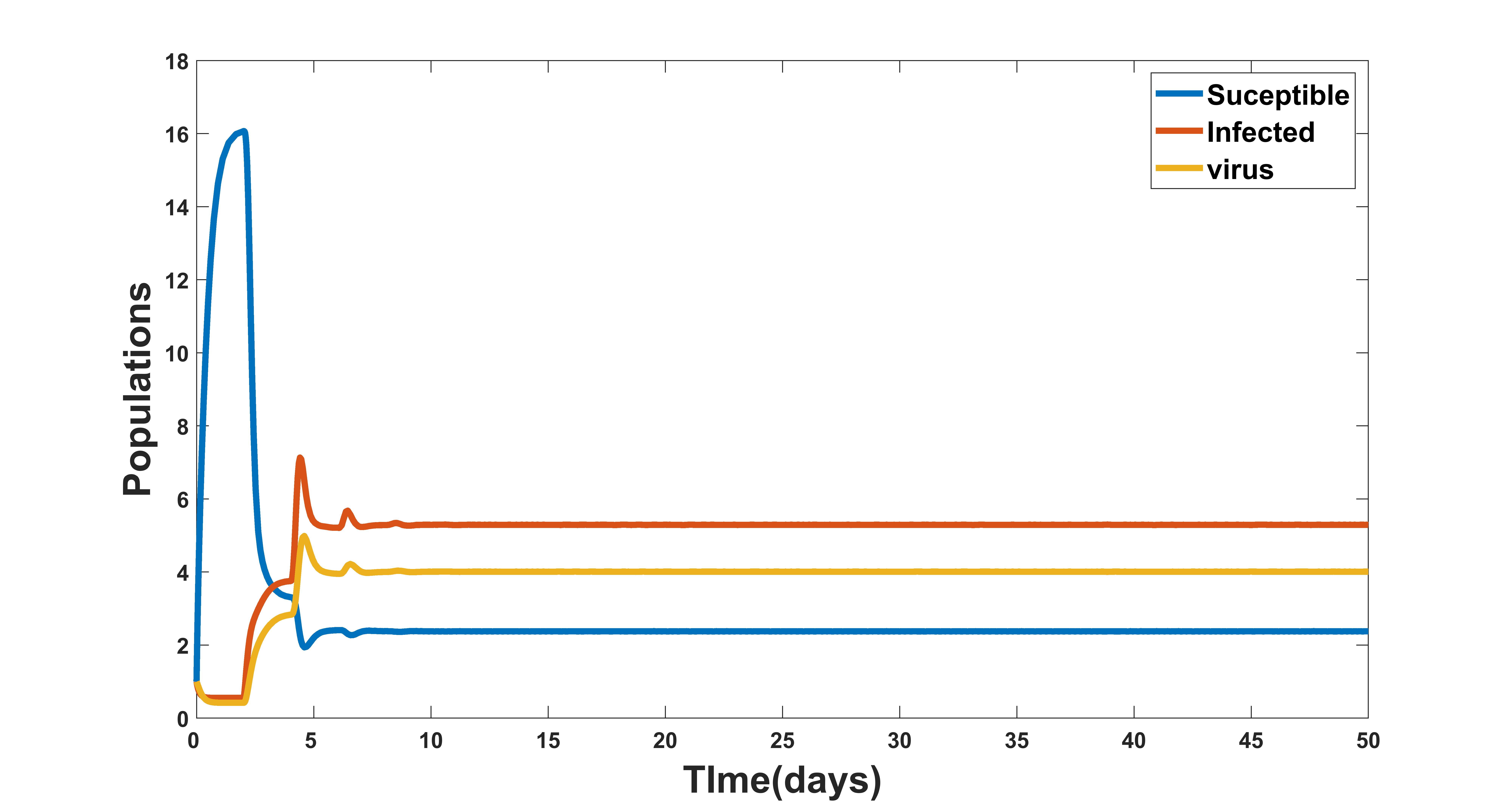}
	\caption{Figure depicting the asymptotic stability of $E_{1}$ (a):$\tau=5$, (b): $\tau=20$ }
	\vspace{6mm}
\end{figure}
 
 \begin{figure}[hbt!]

		\includegraphics[height = 6cm, width = 15.5cm]{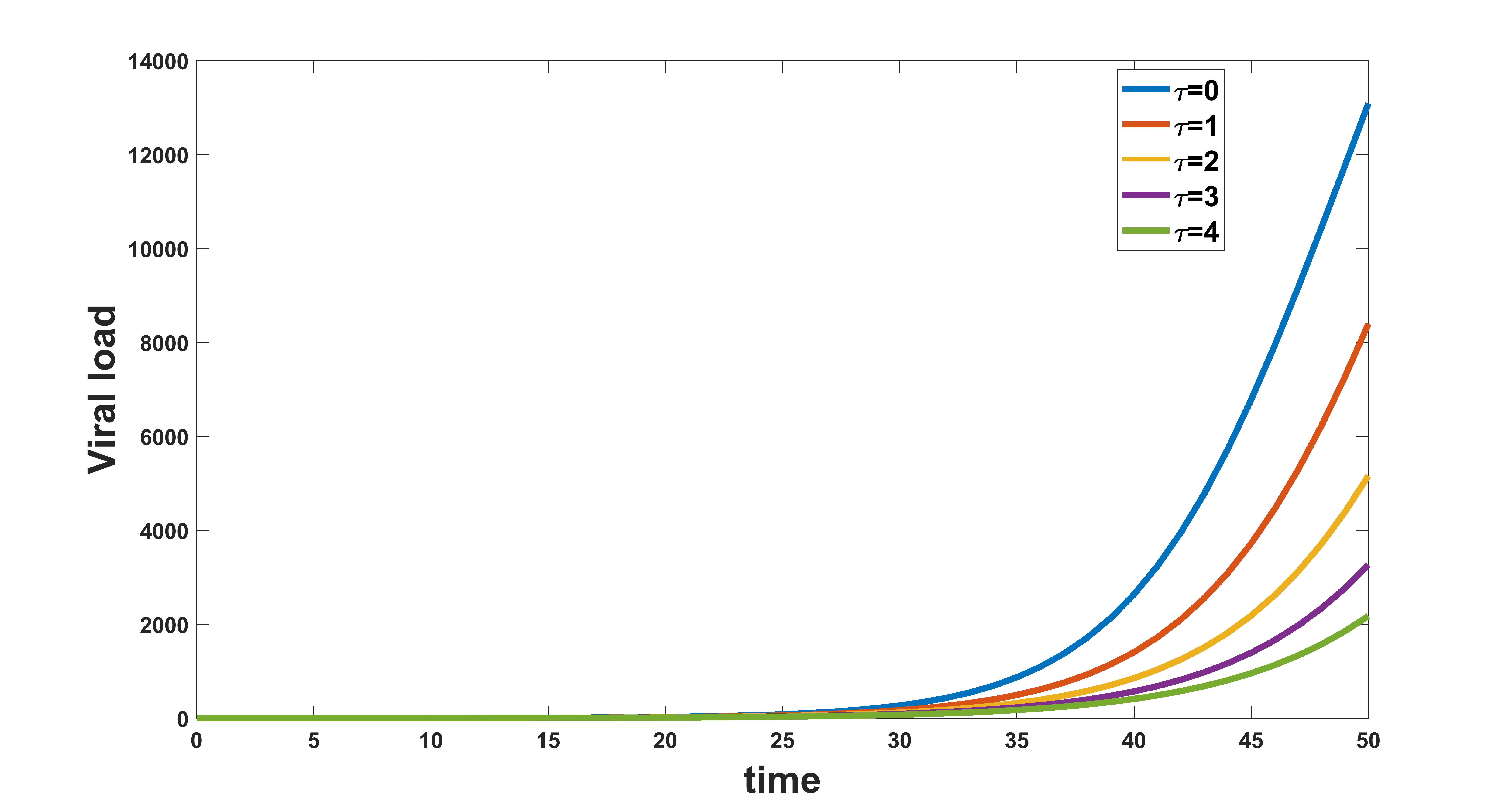}
	\caption{Figure depicting the viral load with different values of intercellular delay }
	\vspace{6mm}
\end{figure}

\newpage
 
\section{Model With Intervention}

Now in the model $(2.1)-(2.3)$ we introduce control variables. We consider two control measures: the first control measure is the antiviral drugs (also called first line drugs) such as Remdesivir, and HCQ  that inhibit viral replication in the body. This is represented by the tuple  $(\mu_{11}, \mu_{21})$  in the model where,  $\mu_{11}$ is the rate at which infected cell decreases because of the administration of antiviral drug and $\mu_{21}$ is the rate at which viral load decreases. The second control measure that we consider is the second line drugs such as methylprednisolone (MP), and dexamethasone. This control is denoted by the tuple $(\mu_{12},\mu_{22})$. The parameter $\epsilon_1=(1-\alpha)$ denotes the efficacy of the antiviral drugs with $\alpha$ denoting the probability of antiviral drugs showing adverse events. When the first line antiviral drug show adverse events $( 0 < \alpha < 1 $), it's quantity is reduced and along with it second line drugs are administered in a  COVID infected individual. The model incorporating these control measures  is given by following system of differential equations.\\

	\begin{eqnarray}
   	\frac{dS}{dt}&=&  \omega \ - \beta SV  - \mu S   \\
   	\frac{dI}{dt} &=&\beta S(t-\tau)V(t-\tau) \ -  x I   \ - \mu I  -\epsilon_1 \mu_{11}(t-\tau_1)I(t-\tau_1)-\alpha \mu_{12}(t)I(t)  \\ 
   	\frac{dV}{dt} &=&  b I  -\epsilon_1 \mu_{21}(t-\tau_1)I(t-\tau_1)-\alpha \mu_{22}(t)V(t)- \mu_{1} V-yV \label{sec2equ3}
   \end{eqnarray} 
   
 \subsection{ \textbf{Optimal Control Problem}}

Let $$U_1=(\mu_{11}, \mu_{21})$$
$$U_2=(\mu_{12},\mu_{22})$$

	Based on above, we now propose and  define the optimal control problem,  with the goal to reduce the cost functional defined as
	
	\begin{equation}
		J(\mu_{11},\mu_{12},\mu_{21},\mu_{22}) = \int_{0}^{T} \bigg(I(t)+V(t)+A_{1} (\mu_{11}^2(t)+\mu_{21}^2(t))+A_{2}(\mu_{12}^2(t)+\mu_{22}^2(t))\bigg) dt   \label{objj}
	\end{equation} 
		subject to the system $(3.1)-(3.3)$	such that $\textbf{u} = (\mu_{11}(t),\mu_{21}(t),\mu_{12}(t),\mu_{22}(t)) \in U$ 	where, $U$ is the set of all admissible controls  given by \\
		$U = \left\{(\mu_{11},\mu_{12},\mu_{21},\mu_{22}) : \mu_{11} \in [0,\mu_{11} max] , \mu_{12} \in [0,\mu_{12} max] , \\mu_{21}\in [0,\mu_{22} max] ,t \in [0,T] \right\}$

The upper bounds of control variables are based on the resource limitation  and  the limit to which these drugs would be prescribed to the patients. Here, the cost function $(\ref{objj})$ represents the number of infected cells and viral load throughout the observation period, and the overall cost for the implementation of each of the interventions. The coefficients $A_1$ and $A_2$ represents the overall cost or effort required for the implementation of first line and second line drugs respectively. Effectively, we want to minimize the infected cells and the viral load in the body with the optimal medication that is also least harmful to the body. Since the drugs administered have multiple effects, the non-linearity for the control variables in the objective become justified \cite{Joshi2002}.
	
	The integrand of the cost function (\ref{objj}), denoted by 
	$$L(I,V,U_1,U_2) = I(t)+V(t)+A_{1}(\mu_{11}(t)^2+\mu_{21}(t^2))+A_{2}(\mu_{12}(t)^2+\mu_{22}(t)^2)$$
	
	is called the Lagrangian or the running cost.
	
	The admissible solution set for the optimal control problem $(3.1)-(3.4)$ is given by
	
	$$\Omega = \left\{ (S, I, V, u)\; | \; S,  I \ and \ V  \text{that satisfy (4.2), } \forall \; \textbf{u} \in U \right\}$$

	

	\subsubsection{\textbf{Existence of Optimal Controls}}\vspace{.25cm}
	
		We will show the existence of optimal control function that minimize the cost function within a finite time span $[0,T]$ by showing that the optimal control problem $(3.1)-(3.4)$ satisfy the conditions stated in Theorem 4.1 of \cite{Wendell}.
	
	\begin{thm}
		There exists a 4-tuple of optimal controls $(\mu_{11}^{*}(t) , \mu_{21}^{*}(t) , \mu_{12}^{*}(t),\mu_{22}^{*}(t))$ in the set of admissible controls U such that the cost functional $(3.4)$ is minimized i.e., 
		
		$$J[\mu_{11}^{*}(t) , \mu_{21}^{*}(t) , \mu_{12}^{*}(t),\mu_{22}^{*}(t)] = \min_{(\mu_{11}^{*} , \mu_{21}^{*} , \mu_{12}^{*}, \mu_{22}^{*}) \in U} \bigg \{ J[\mu_{21},\mu_{11},\mu_{21},\mu_{12},\mu_{22}]\bigg\}$$ 
		corresponding to the optimal control problem $(3.1)-(3.4)$.
	\end{thm}

	\begin{proof}
		
		 In order to show the existence of optimal control functions, we will show that the following conditions are satisfied : 
		
		\begin{enumerate}
			\item  The solution set for the system $(3.1)-(3.3)$ along with bounded controls must be non-empty, $i.e.$, $\Omega \neq \phi$.
			
			\item  U is closed and convex and system $(3.1)-(3.3)$ should be expressed linearly in terms of the control variables with coefficients that are functions of time and state variables.
			
			\item The Lagrangian L should be convex on U and $L(I,V,\mu_{11},\mu_{12},\mu_{21},\mu_{22}) \geq g(\mu_{11},\mu_{21},\mu_{12},\mu_{22})$, where $g(\mu_{11},\mu_{12},\mu_{21},\mu_{22})$ is a continuous function of control variables such that $|(\mu_{11},\mu_{12},\mu_{21},\mu_{22})|^{-1}$ $g(\mu_{11},\mu_{12},\mu_{21},\mu_{22})\to \infty$ whenever  $|(\mu_{11},\mu_{21},\mu_{12},\mu_{22})| \to \infty$, where $|.|$ is an $l^2(0,T)$ norm.
		\end{enumerate}

		Now we will show that each of the above conditions are satisfied : 
		
		1. From positivity and boundedness of solutions of the system $(3.1)-(3.3)$, we know that all solutions of the system are positive and bounded for each bounded control variable in $U$. To show $\Omega \neq \phi$ we will first show that RHS of (3.1)-(3.3) satisfies Lipschitz condition with respect to the state variables and use  Picard-Lindelof Theorem\cite{makarov2013picard} to show $\Omega \neq \phi$.

		Let 
		
\begin{equation*}
	X=
\begin{pmatrix}
S \\
I \\
V
\end{pmatrix}
\end{equation*}

The RHS of the system $(3.1)-(3.3)$ can be written as,
$$\frac{dX}{dt}=AX+F(X,X_{\tau},X_{\tau_{1}}))=G(X,X_{\tau},X_{\tau_{1}})$$
where\\
\begin{equation*}
A=
\begin{pmatrix}
-\mu & 0 & 0 \\
0 & -(x+\mu) & 0 \\
0 & b & -(y+\mu_{1})
\end{pmatrix}
\end{equation*} \\

\begin{equation*}
F=
\begin{pmatrix}
\omega-\beta S V \\
\beta S_\tau V_\tau -\epsilon \mu_{11\tau_1}I_{\tau_1}-\alpha \mu_{12}(t)I(t)\\
-\epsilon \mu_{21\tau_1}V_{\tau_1} - -\mu_{22}(t)V(t)

\end{pmatrix}
\end{equation*} 
Here,
$$X_{\tau}=X(t-\tau)$$
$$X_{\tau_1}=X(t-\tau_1)$$
$$\mu_{i1\tau_{1}}=\mu_{i1}(t-\tau_1)$$ 

Now \\
\begin{equation*}
\begin{split}
|F(X_1,X_{1\tau},X_{1\tau_1})-F(X_2,X_{2\tau},X_{2\tau_1})| &\le
M_1|X_1-X_2|+M_2|X_{1\tau}-X_{2\tau}|+M_3|X_{1\tau_1}-X_{2\tau_1}| \\\\
|G(X_1,X_{1\tau},X_{1\tau_1})-G(X_2,X_{2\tau},X_{2\tau_1})| & \le
|A(X_1-X_2)|+M_1|X_1-X_2|+M_2|X_{1\tau}-X_{2\tau}|+M_3|X_{1\tau_1}-X_{2\tau_1}|\\\\
 & \leq M\bigg(|X_1-X_2|+|X_{1\tau}-X_{2\tau}|+|X_{1\tau_1}-X_{2\tau_1}|\bigg)\\
 \end{split}
 \end{equation*}

where $$M =max(|A|+ M_1, M_2, M_3)$$ and $M_i$ are independent of $S,I,V$.

		Hence,the right hand side of the system $(3.1)-(3.3)$ satisfies Lipschitz condition with respect to state variables.	Now, using the positivity and boundedness of the  solution of system $(3.1)-(3.3)$ and Picard-Lindelof Theorem\cite{makarov2013picard}, we have satisfied condition 1.
		
		2. The contro, set $U$ is closed and convex by definition. Also, the system $(3.1)-(3.3)$ is clearly linear with respect to each control variable such that coefficients are only state variables or functions dependent on time. Hence condition 2 is satisfied.
		
		3. Choosing $g(\mu_{11},\mu_{12},\mu_{21},\mu_{22}) = c(\mu_{11}^2+\mu_{21}^2+\mu_{12}^2+\mu_{22}^2)$ such that $c = min\left\{A_{1},A_{2}\right\}$, condition 3 is satisfied.
		
		Thus we have shown that all the conditions stated  in Theorem 4.1 of \cite{Wendell} is satisfied. Hence there exists optimal controls, a 4-tuple $(u_{11}^{*},u_{12}^{*},u_{21}^{*},u_{22}^{*})\in U$ that minimizes the cost function (\ref{objj}) subject to system $(3.1)-(3.3)$.
	\end{proof}

	\subsubsection{\textbf{Characterization of Optimal Controls}}
    In this section we obtain the optimal control solutions of the optimal control problem $(3.1)-(3.4)$ using Pontryagin's Maximum Principle  \cite{gollmann2009optimal}.
	
	The Hamiltonian for this problem is given by 
	$$H(S,I,V,U_1,U_2,\lambda) := L(I,V,U_1,U_2) + \lambda_{1} \frac{\mathrm{d} S}{\mathrm{d} t} +\lambda _{2}\frac{\mathrm{d} I}{\mathrm{d} t}+ \lambda _{3} \frac{\mathrm{d} V}{\mathrm{d} t}$$
	
		Here $\lambda$ = ($\lambda_{1}$,$\lambda_{2}$,$\lambda_{3}$) is called co-state vector or adjoint vector.
	
	Now the Canonical equations that relate the state variables to the co-state variables are  given by 
	
	\begin{equation}
	\begin{aligned}
	 \frac{\mathrm{d} \lambda _{1}}{\mathrm{d} t} &= -\frac{\partial H}{\partial S}-\chi_{[0,T-\tau]}(t)\frac{\partial H(t+\tau)}{\partial S(t-\tau)}\\
	 \frac{\mathrm{d} \lambda _{2}}{\mathrm{d} t} &= -\frac{\partial H}{\partial I}-\chi_{[0,T-\tau_1]}(t)\frac{\partial H(t+\tau_1)}{\partial I(t-\tau_1)}\\
	 \frac{\mathrm{d} \lambda _{3}}{\mathrm{d} t} &= -\frac{\partial H}{\partial V}-\chi_{[0,T-\tau]}(t)\frac{\partial H(t+\tau)}{\partial V(t-\tau)}-\chi_{[0,T-\tau_1]}(t)\frac{\partial H(t+\tau_1)}{\partial V(t-\tau_1)}
	\end{aligned}
	\end{equation}

	
	Substituting the Hamiltonian in $(3.5)$ we get the following  canonical system 
	\begin{equation}
	\begin{aligned}
	\frac{\mathrm{d} \lambda _{1}}{\mathrm{d} t} &= \lambda _{1}(\beta V+\mu)-\chi_{[0,T-\tau]}(t)\lambda_2(t+\tau)\beta V\\
	\frac{\mathrm{d} \lambda _{2}}{\mathrm{d} t} &= -1 + \lambda _{2}(x+\mu+\alpha \mu_{12}(t))-\lambda _{3}b + \chi_{[0,T-\tau_1]}(t)\lambda_2(t+\tau_1)\epsilon_1\mu_{11}(t) \\
	\frac{\mathrm{d} \lambda _{3}}{\mathrm{d} t} &= -1+\lambda _{1}\beta S+\lambda _{3} (\mu _{1}+y+\alpha \mu_{22})-\chi_{[0,T-\tau]}(t)\lambda_2(t+\tau)\beta S(t)+\chi_{[0,T-\tau_1]}(t)\lambda_3(t+\tau_1)\epsilon_1 \mu_{21}(t) 
	\end{aligned}
	\end{equation}
	along with transversality conditions
	$ \lambda _{1} (T) = 0, \  \lambda _{2} (T) = 0, \  \lambda _{3} (T) = 0. $
	
	Now, to obtain the optimal controls, we will use the Hamiltonian minimization condition \cite{gollmann2009optimal} given by
		\begin{equation}
	    \frac{\partial H}{\partial u(t)}+ \frac{\partial H(t+\tau_1)}{\partial u(t-\tau_1)}=0
		\end{equation}
		
	Using $(3.7)$ and differentiating the Hamiltonian with respect to each of the controls and solving the equations, we obtain the optimal controls as 
	
	\begin{eqnarray*}
	\mu_{11}^{*} &=& \min\bigg\{ \max\bigg\{\frac{\chi_{[0,T-\tau_1]}(t)\epsilon_1 \lambda _{2}(t+\tau_1)I(t)}{2A_{1}},0 \bigg\}, \mu_{11}max\bigg\}\\
	\mu_{21}^{*} &= &\min\bigg\{ \max\bigg\{\frac{\chi_{[0,T-\tau_1]}(t)\epsilon_1 \lambda _{3}(t+\tau_1)V(t)}{2A_{1}},0 \bigg\}, \mu_{21}max\bigg\}\\
	\mu_{12}^{*}& = &\min\bigg\{ \max\bigg\{\frac{\alpha \lambda _{2}I(t)}{2A_{2}},0 \bigg\}, \mu_{12}max\bigg\}\\
	\mu_{22}^{*}& = &\min\bigg\{ \max\bigg\{\frac{\alpha \lambda _{3}V(t)}{2A_{2}},0 \bigg\}, \mu_{12}max\bigg\}\\
		\end{eqnarray*}
 
 \subsubsection{\textbf{Numerical Illustration of Optimal Control Problem}}
 	In this section, we perform numerical simulations to understand the efficacy of first line and second line drug interventions. This is done by studying the effect of control  on the  dynamics of the system $(3.1)-(3.3)$. Let there exist a step size $h > 0$ and $n > 0$ such that $T-t_0=nh$ and $\tau=m_1h$, and $\tau_1=m_2 h$. Let $m= \max(m_1, m_2)$. For programming point of view we consider m knots to left of $t_0$  and right of T and we obtain
the following partition:

$\Delta= \bigg(t_{-m}=-\max(\tau,\tau_1,\tau_2).... < t_1 < t_0 =0 < t_1 ...<t_n = t_f(=T) < ....<t_{n+m} )\bigg)$.\\
Then, we have $t_i=t_0+ih (-m\leq i\leq n+m)$. Now we define the state $S,I,V$, the adjoint vectors  $\lambda_1,\lambda_2,\lambda_3$  and controls $\mu_{11},\mu_{21},\mu_{12},\mu_{22}$ in terms of nodal points $S_i, I_i, V_i$ , $\lambda_1^i, \lambda_2^i, \lambda_3^i$ and $\mu_{11}^i,\mu_{21}^i,\mu_{12}^i,\mu_{22}^i$. For simulation we use combination of forward and backward
difference approximations and the parameter values are taken from \cite{chhetri2020within} which is given in table 4. The value of $\tau $ is taken as $1$ day \cite{bar2020science} and the value of $\tau_1$ is approximated and taken as $6$ day \cite{conticini2020high}. The probability of antiviral drugs showing adverse events is taken as $0.6$ \cite{grein2020compassionate}. The positive weights chosen for objective coefficients are $A_{1}$ = 500, $A_{2}$ = 200. Since second line drugs (methylprednisolone or dexamethasone) are cheaper and easily available, the value of the coefficient corresponding to the second line drug ($A_2$) is taken lesser than that of the coefficient corresponding to antiviral drug ($A_1$). \\

\begin{table}[ht!]
	\caption{Parameter Values} 
	
	\begin{center}
		\begin{tabular}{|c|c|c|c|c|c|c|c|c|c|c|c|c|c|c|c|c|c|c|c|c|}
			\hline
			$\omega$ & $\beta$ & $\mu$ & $\mu_{1}$ & $b$ & $d_{1}$ & $d_{2}$& $d_{3}$ & $d_{4}$ & $d_{5}$ & $d_{6}$ &  $b_{1}$ & $b_{2}$ & $b_{3}$ & $b_{4}$ & $b_{5}$ & $b_{6}$&$\alpha$ & $\tau$ & $\tau_1$\\
			\hline
			  10 & 0.05& .5 & 1.1 & .5 & 0.027& 0.22 &0.1 &0.428 & 0.01 & 0.01 & 0.1& .1& 0.08 & .11 & .1 & .07& 0.6 &1 & 6\\
			\hline

		\end{tabular}
	\end{center}
\end{table}

In figure 4, we evaluate the role of optimal drug interventions in reducing the infected cell population. The infected cell population is plotted over $40$ days considering individual role and combined role of the optimal drug interventions. We see that the infected cell count increases exponentially when no controls are considered, whereas  the infected cell count starts decreasing as optimal controls are considered with maximum decrease when both the first line and second line drugs are considered together. The viral population is plotted in figure 5 and compared with the no intervention case, only first line drug case, and combined first and second line drug case. It is seen that the viral load decreases maximum when both the first line and second line drugs are considered together. In figure 6 the profile of optimal cost generated for first line drug, and combined first and second line drugs is plotted. We notice that when both the controls (first line and second line drugs) are considered together, the generated cost which is the weighted sum of cost due to disease and controls is minimal (green color curve).

The comparative study done here,  simulated in figure (4, 5, 6)  suggests that when the first line antiviral drugs starts showing adverse events, considering the antiviral drugs in reduced quantity along with the second line drug is highly effective  in reducing the infected cell and viral load in a COVID infected patients than considering only the first line  drug. 

 \begin{figure}[hbt!]
	\includegraphics[height = 6cm, width = 15.5cm]{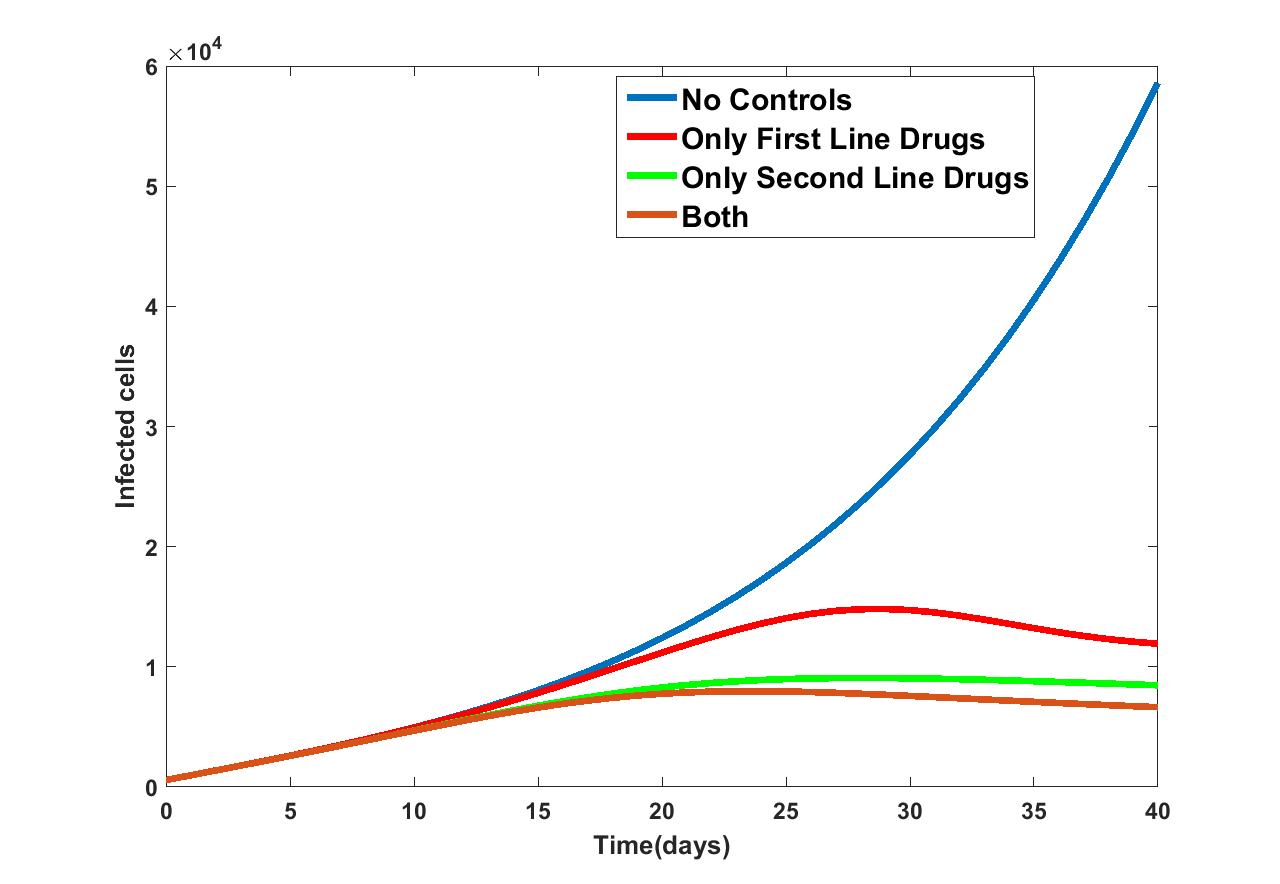}
	\caption{Figure depicting the infected cell population under different optimal intervention strategy }
	\vspace{6mm}
\end{figure}
 
 \begin{figure}[hbt!]

		\includegraphics[height = 6cm, width = 15.5cm]{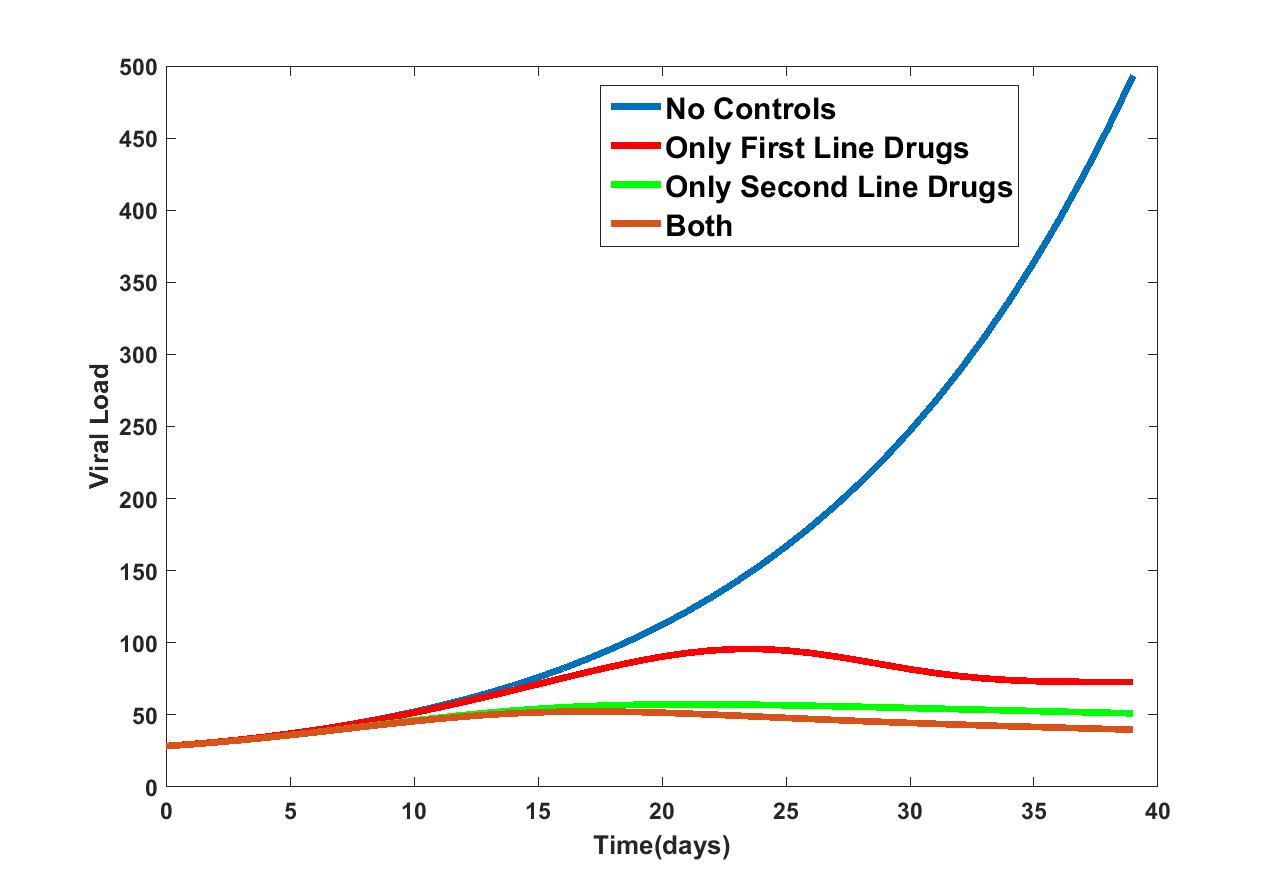}
	\caption{Figure depicting the viral load under different optimal intervention strategy }
	\vspace{6mm}
\end{figure}

\begin{figure}[hbt!]

		\includegraphics[height = 6cm, width = 15.5cm]{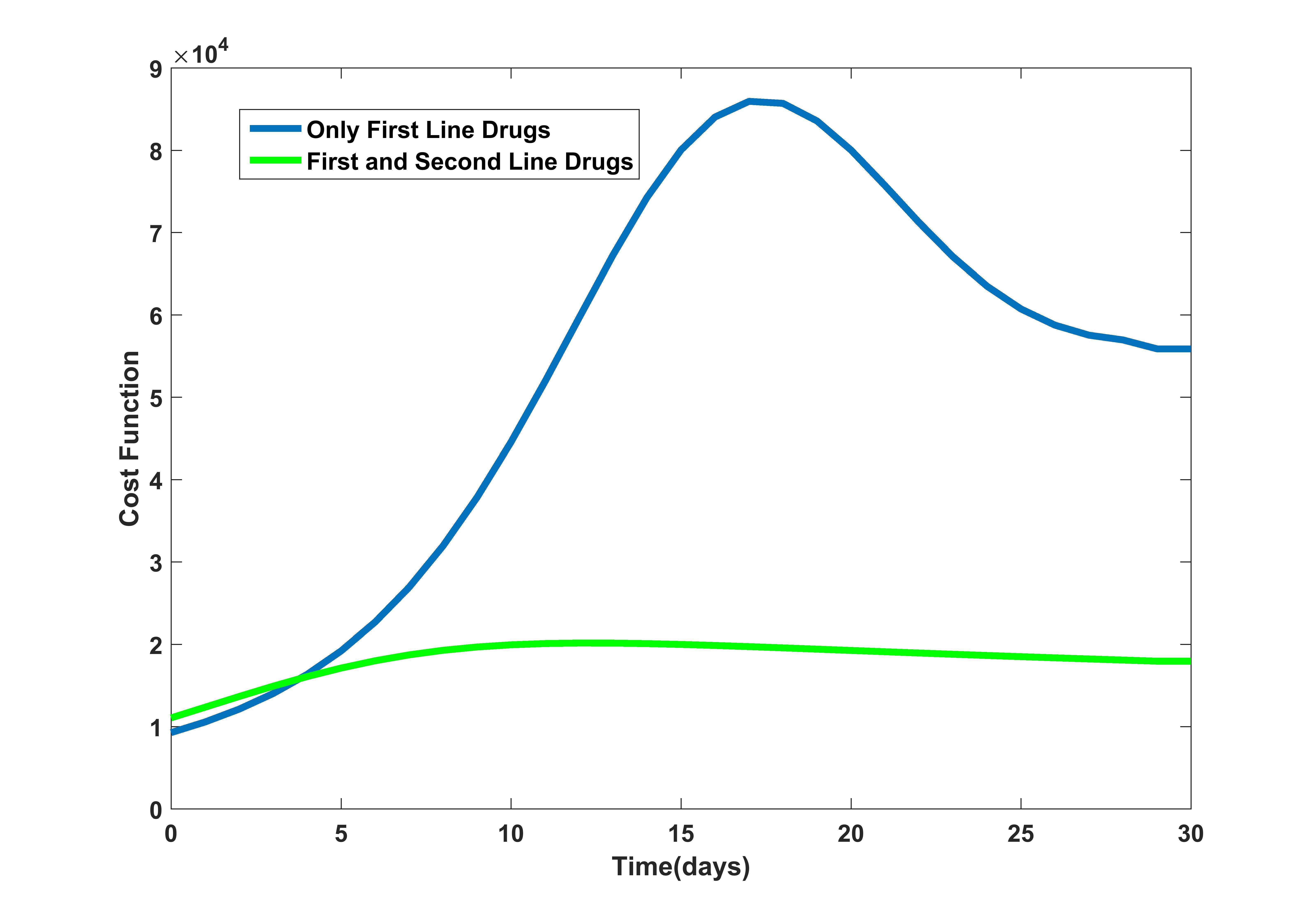}
	\caption{Figure depicting optimal cost under different optimal intervention strategy }
	\vspace{6mm}
\end{figure}

\newpage
\subsection{\textbf{Time Optimal Control Problem}}
In this section we formulate a time optimal control problem for the system $(3.1)-(3.3)$. Our objective here is to drive the system from any initial state to the desired infection free equilibrium state $E_0=(\frac{\omega}{\mu},0,0)$ in minimum time using the optimal controls.

\subsubsection{\textbf{Formulation of Time Optimal Control Problem and Existence of Optimal Control}}

Let $U$ denote the control set given by,

$U = \left\{(\mu_{11},\mu_{12},\mu_{21},\mu_{22}) : \mu_{11} \in [0,\mu_{11} max] , \mu_{12} \in [0,\mu_{12} max] , \mu_{21}\in [0,\mu_{22} max] ,t \in [0,T] \right\}$

The time optimal control problem (a Mayer problem of Optimal Control \cite{pierre1986optimization} ) with  $U$ as a control space is given by,

\begin{equation}
		\min \hspace{.2cm} J(\mu_{11},\mu_{12},\mu_{21},\mu_{22}) = \int_{0}^{T} 1 dt = \min_{u \in U} \hspace{.2cm} T   \hspace{1cm}\label{obj}
	\end{equation}
subject to:
	\begin{eqnarray}
   	\frac{dS}{dt}&=&  \omega \ - \beta SV  - \mu S   \\
   	\frac{dI}{dt} &=&\beta S(t-\tau)V(t-\tau) \ -  x I   \ - \mu I  -\epsilon_1 \mu_{11}(t-\tau_1)I(t-\tau_1)-\alpha \mu_{12}(t)I(t)  \\ 
   	\frac{dV}{dt} &=&  b I  -\epsilon_1 \mu_{21}(t-\tau_1)V(t-\tau_1)-\alpha \mu_{22}(t)V(t)- \mu_{1} V-yV \label{sec2equ3}
   \end{eqnarray} 
$(S(0), I(0),V(0))=(S_0, I_0, V_0)$ and $(S(T), I(T),V(T))=(\frac{\omega}{\mu}, 0, 0)$

Let $X(t)=(S(t),I(t),V(t))$ and set A be the subset of $tX$-space ($R^{1+3}$) ie. $A \subset R^{1+3}$ from where we get the state variables (S, I, V).

The set of all admissible solutions to the time optimal control problem is given by,

	$$\Omega = \left\{ (X(t), u(t))\; | \text{X satisfies system} \hspace{.15cm}    (3.9)-(3.11)  \hspace{.3cm} \forall u \in U \right\}$$
	
	 Now from the admissible solution set $\Omega$, we wish to find a solution that minimizes the time to reach the terminal state $(S(T), I(T), V(T))$ from any given initial state $(S_0, I_0, V_0)$. In the following theorem we prove the  existence of an optimal control  using Filippov's existence theorem.
	
	\begin{thm}
	      There exists an optimal controls $(\mu_{11}^*,\mu_{12}^*,\mu_{21}^*,\mu_{22}^*)$ that drives the system from any given initial state $(S_0, I_0, V_0)$ to the desired infection free state $(S(T), I(T), V(T))$ in minimum time for the time optimal control problem $(3.8)-(3.11)$ provided $\Omega \neq \phi$.
	    
	    \begin{proof}
	        For proving the existence of an optimal controls using Filippov's existence theorem it is enough to show that the following conditions are satisfied:
	        
	        1. The set A is compact 
	        
	       2. U is compact and convex
	        
	        3. The set of boundary points $B= \{0, S_0, I_0, V_0, T, S(T), I(T), V(T)\}$ is compact and objective function is continuous on B.
	        
	        We will now show that all the above conditions are satisfied by the time optimal control problem $(3.8)-(3.11)$.
	        
	        (i). From the positivity and boundedness section we know that if the initial values of the state variables are positive then  all the solutions of the system with each bounded controls remains  positive and bounded for all time. This proves that the set A is compact. 
	        
	        (ii). By the definition the control set $U$ is compact and convex. 
	        
	        (iii). Condition 3 is satisfied by the definition of $B= \{0, S_0, I_0, V_0, T, S(T), I(T), V(T)\}$ and the definition of the objective function $J=T$.
	        
	         \end{proof}
	\end{thm}
\subsubsection{\textbf{Characterization of Optimal Controls}}
Given the time optimal control problem $(3.8)-(3.11)$ we apply Pontryagin's Minimum Principle for the characterization of the optimal controls. 

\begin{thm}(\textbf{Pontryagin's Minimum Principle for linear control problem (\cite{pontryagin2018mathematical,gollmann2009optimal}}))
Suppose that $U^*=(\mu_{11}^*, \mu_{12}^*,\mu_{21}^*, \mu_{22}^*)$ is a minimizer for the time optimal control problem $(3.8)-(3.11)$ and  $X(t)^*=(S(t)^*, I(t)^*, V(t)^*)$ denote the optimal solution of the system $(3.9)-(3.11)$. Then there exists a piecewise $C^1$ vector function $\lambda^*(t)=(\lambda_1^*(t),\lambda_2^*(t),\lambda_3^*(t)) \neq 0 $ such that 

$$ \frac{\mathrm{d} \lambda^*}{\mathrm{d} t} = -\frac{\partial H}{\partial X}-\chi_{[0,T-\tau]}(t)\frac{\partial H(t+\tau)}{\partial X(t-\tau)}-\chi_{[0,T-\tau_1]}(t)\frac{\partial H(t+\tau_1)}{\partial X(t-\tau_1)}$$ 
where $H$ is the Hamiltonian function defined as $H(X,U,\lambda)=1+\lambda_1 \frac{d S}{dt}+\lambda_2 \frac{d I}{dt}+\lambda_3 \frac{d V}{dt}$ and:

1. the function $H(X^*, u, \lambda^* )$ attains a minimum on $U$
$$H(X^*, u^*, \lambda^*) \leq H(X^*, u, \lambda^*) \forall u \in U$$
2. the Hamiltonian is constant equal to zero along the optimal solution:
$$H(X^*, u^*, \lambda^*)=0$$
 where $u^* = (\mu_{11}^*, \mu_{12}^*, \mu_{21}^*, \mu_{22}^* )$
\end{thm}
 
 The Hamiltonian for the problem $(3.8)-(3.11)$ is given by 
	$$H(S,I,V,U_1,U_2,\lambda) := 1 + \lambda_{1} \frac{\mathrm{d} S}{\mathrm{d} t} +\lambda _{2}\frac{\mathrm{d} I}{\mathrm{d} t}+ \lambda _{3} \frac{\mathrm{d} V}{\mathrm{d} t}$$

	Using Minimum Principle we have the following  Canonical equations that relate the state variables and the co-state variables 
	
	\begin{equation}
	\begin{aligned}
	\frac{\mathrm{d} \lambda _{1}}{\mathrm{d} t} &= \lambda _{1}(\beta V+\mu)-\chi_{[0,T-\tau]}(t)\lambda_2(t+\tau)\beta V\\
	\frac{\mathrm{d} \lambda _{2}}{\mathrm{d} t} &= \lambda _{2}(x+\mu+\alpha \mu_{12}(t))-\lambda _{3}b + \chi_{[0,T-\tau_1]}(t)\lambda_2(t+\tau_1)\epsilon_1\mu_{11}(t) \\
	\frac{\mathrm{d} \lambda _{3}}{\mathrm{d} t} &= \lambda _{1}\beta S+\lambda _{3} (\mu _{1}+y+\alpha \mu_{22})-\chi_{[0,T-\tau]}(t)\lambda_2(t+\tau)\beta S(t)+\chi_{[0,T-\tau_1]}(t)\lambda_3(t+\tau_1)\epsilon_1 \mu_{21}(t) 
	\end{aligned}
	\end{equation}

 Now from condition 1 of Theorem 3.3 (Hamiltonian Minimization condition) we have 
 $$H(X^*, u^*, \lambda^*) \leq H(X^*, u, \lambda^*) \forall u \in U$$
 The optimal Hamiltonian function is given by,
 \begin{equation*}
 \begin{split}
 H(X, u, \lambda) & = 1 + \lambda_1(t)\bigg( \omega \ - \beta SV  - \mu S\bigg)  \\[4pt]
 & + \lambda_2(t)\bigg(\beta S(t-\tau)V(t-\tau) \ -  x I   \ - \mu I  -\epsilon_1 \mu_{11}(t-\tau_1)I(t-\tau_1)-\alpha \mu_{12}(t)I \bigg)\\
 & + \lambda_3(t)\bigg(b I(t)  -\epsilon_1 \mu_{21}(t-\tau_1)V(t-\tau_1)-\alpha \mu_{22}(t)V- \mu_{1} V-yV\bigg)
  \end{split}
 \end{equation*}
 
 From these Hamiltonian equation we see that the minimization of Hamiltonian function depends on the extreme values of the control functions and the sign of $\lambda_2 I(t-\tau_1)$, $\lambda_3 V(t-\tau_1)$, $\lambda_2I(t)$ and $\lambda_3 V(t)$. Therefore, we conclude that the optimal controls might be of bang-bang type provided singular solution does not exists. The optimal control functions would like: 
 
 \begin{equation}
    \mu_{11}^* = 
    \begin{cases}
     \ 0,  &\text{if \quad}   \epsilon_1 \lambda_2(t)  I(t-\tau_1) < 0\\\\
     \mu_{11 max},  &\text{if \quad}   \epsilon_1 \lambda_2(t) I(t-\tau_1) > 0\\\\
      \mathord{?},  &\text{if \quad}   \epsilon_1 \lambda_2(t) I(t-\tau_1) =0 
     \end{cases}
 \end{equation}
 
 \begin{equation}
    \mu_{21}^* = 
    \begin{cases}
     \ 0,  &\text{if \quad}   \epsilon_1 \lambda_3(t) V(t-\tau_1) < 0\\\\
     \mu_{21 max},  &\text{if \quad}  \epsilon_1 \lambda_3(t) V(t-\tau_1) > 0\\\\
     \mathord{?},  &\text{if \quad}   \epsilon_1 \lambda_3  V(t-\tau_1) =0 
     \end{cases}
 \end{equation}
 
 \begin{equation}
    \mu_{12}^* = 
    \begin{cases}
     \ 0,  &\text{if \quad}   \alpha \lambda_2(t) I(t) < 0\\\\
     \mu_{12 max},  &\text{if \quad}    \alpha \lambda_2(t) I(t) > 0 \\\\
     \mathord{?},  &\text{if \quad}  \alpha \lambda_2(t) I(t) =0
     \end{cases}
 \end{equation}
 
  \begin{equation}
    \mu_{22}^* = 
    \begin{cases}
     \ 0,  &\text{if \quad}   \alpha \lambda_3(t) V(t) < 0\\\\
     \mu_{22 max},  &\text{if \quad}   \alpha  \lambda_3(t) V(t)  > 0\\\\
     \mathord{?},  &\text{if \quad}  \alpha \lambda_3(t) V(t) =0
     \end{cases}
 \end{equation}

   Now in order to show that the optimal controls are of bang-bang type we must show that the solution does not exhibit singular arc in some interval $B \subset [0, T]$.  Now in the following we will show that solution does not exhibit singular arc in interval $B \subset [0, T]$. Since $\alpha > 0 $ and $\epsilon_1 > 0$,  the maximum or minimum value of the optimal controls depends on the sign of $\lambda_2(t) I(t)$ and $\lambda_3(t)V(t)$.\\ 
   Let
   $$\phi_1(t)= \lambda_2(t) I(t)$$
    $$\phi_2(t)= \lambda_3(t) V(t)$$
    $$\phi(t) = (\phi_1(t), \phi_2(t))$$
    The derivative of $\phi_1(t)$ and $\phi_2(t)$ is given by,
     $$\dot \phi_1(t)= \lambda_2(t) \dot I(t) + \dot \lambda_2(t) I(t)$$
    $$\dot \phi_2(t)= \lambda_3(t) \dot V(t) + \dot \lambda_3(t) V(t)$$
 
  Now let's assume that singular solution exists, that is $\phi(t)=0$ on interval $B \subset [0, T]$. This would mean $\phi_1(t)=0$ and $\phi_2(t)=0$ on interval $B \subset [0, T]$. This implies that $\lambda_2(t)$ and $\lambda_3(t)$ both are zero in B since $I(t) > 0$ and $V(t)> 0$ $\forall t \in [0, T]$. Now $\phi_1(t)=0$ and $\phi_2(t)=0$ on interval $B$ implies that $\phi_1^{'}(t)=0$ and $\phi_2^{'}(t)=0$. From the definition of the derivative of $\phi_2^{'}(t)$, 
 $\phi_2^{'}(t)=0$ would imply that $\lambda_3^{'}(t)=0$ in $B$. Substituting $\lambda_3^{'}(t)=0$ in the canonical equations $(3.12)$ we find that $\lambda_1(t)=0$ in $B$. Therefore,  in interval $B \subset [0, T]$ we see that the adjoint vector $\lambda=(\lambda_1,\lambda_2,\lambda_3)=0$. This is a contradiction with Theorem 3.3, as adjoint variables $\lambda_1$, $\lambda_2$ and $\lambda_3$ cannot vanish simultaneously. Therefore, the switching functions $\phi_1(t)$ and $\phi_2(t)$ vanishes only at isolated points. As a consequences, the controls can assume two values: 0 and maximum value. The switching times are defined as the time instants at which the switching functions $\phi_1(t)$ and $\phi_2(t)$ changes its sign. Therefore, two types of switch can occur: one when the optimal control changes its value from zero to maximum and the other when optimal control changes its value from  maximum to zero.
 
 We summarize the above discussion in the following result.\\
 
 \begin{thm}
 The optimal control solution for the time optimal control problem $(3.8)-(3.11)$ is  a bang-bang type with possibility of switches occurring in the optimal trajectory. The optimal controls is given by,
 
  \begin{equation}
    \mu_{11}^* = 
    \begin{cases}
     \ 0,  &\text{if \quad}   \lambda_2(t) I(t-\tau_1) < 0\\\\
     \mu_{11 max},  &\text{if \quad}   \lambda_2(t) I(t-\tau_1) > 0
     \end{cases}
 \end{equation}
 
 \begin{equation}
    \mu_{11}^* = 
    \begin{cases}
     \ 0,  &\text{if \quad}   \lambda_3(t) V(t-\tau_1) < 0\\\\
     \mu_{21 max},  &\text{if \quad}   \lambda_3(t) I(t-\tau_1) > 0
     \end{cases}
 \end{equation}
 
 \begin{equation}
    \mu_{12}^* = 
    \begin{cases}
     \ 0,  &\text{if \quad}    \lambda_2(t) I(t) < 0\\\\
     \mu_{12 max},  &\text{if \quad}    \lambda_2(t) I(t) > 0
     \end{cases}
 \end{equation}
  \begin{equation}
    \mu_{22}^* = 
    \begin{cases}
     \ 0,  &\text{if \quad}    \lambda_3(t) V(t) < 0\\\\
     \mu_{22 max},  &\text{if \quad}     \lambda_3(t) V(t)  > 0
     \end{cases}
 \end{equation}
 
 \end{thm}
 
 In the context of time optimal control problem $(3.8)-(3.11)$, we wish to reach the desired terminal state ( which is the infection free equilibrium state in our case) from a given initial state in minimal time. The following results gives the existence of optimal solution and the characterisation of optimal controls at the terminal time. In the following lemma we show that at the terminal time the optimal controls assumes its maximum value and later using this lemma we prove that the optimal control is maximum on the entire $[0,T]$ with some condition on the control variables. 
 
 For the sake of simplicity let us denote all the control variable by a vector $u(t)$ that is,  $u(t)=(\mu_{11}(t),\mu_{12}(t),\mu_{21}(t),\mu_{22}(t))$
 
 \begin{lem}
 Let $\dot V(T) = -y$ where, $y > 0$.  The time optimal control problem $(3.8)-(3.11)$ with $(S(T)=\frac{\omega}{\mu}, I(T)=0, V(T)=0)$ admits an optimal solution. Moreover $\textbf{u}_\text{opt}(T)= \textbf{u}_\text{max}$ provided $0<\lambda_3(T) < \frac{1}{y}$
 \begin{proof}
   From subsection 3.1.1, we know  that 	the right hand side of the system $(3.9)-(3.11)$ satisfies Lipschitz condition with respect to state variables and also the solutions of system $(3.9)-(3.11)$ are positive and bounded. Therefore, solution exists for the system $(3.9)-(3.11)$ for each bounded control variables \cite{makarov2013picard}.  This implies that the admissible solution set $\Omega \neq \phi$ and thus using existence theorem we conclude that the optimal control problem $(3.8)-(3.11)$ with $(S(T)=\frac{\omega}{\mu}, I(T)=0, V(T)=0)$ admits an optimum. 
     
     Now from condition 2 of Pontryagin's Minimum Principle (Theorem 3.3) we have,
     $$H(S(T),I(T),V(T),u^*, \lambda(T))=0$$
     substituting the definition of $H$ in the above equation we get,
     \begin{equation}
         \lambda_2(T) \frac{d I}{d t} +\lambda_3(T) \frac{d V}{d t} =-1
     \end{equation}
     since the initial values of $I(t)$ and $V(t)$ are both positive and  $I(T)$ and $V(T)$ are both zero, we have ${\dot I(T) < 0}$ and $\dot V(T) < 0$. Now from equation $(3.21)$ we get,
     $$\lambda_2(T)=\frac{-1-\lambda_3(T) \dot V(T)}{\dot I(T)}$$
    Let ${\dot I(T) =-x}$ and ${\dot V(T) =-y}$ where $x >0$ and $y > 0$. Assuming  $0<\lambda_3(T) < \frac{1}{y}$ we have $\lambda_2(T) > 0$. Hence at the terminal time we have $\textbf{u}_ \text{opt}(T)= \textbf{u}_\text{max}$.
 \end{proof}
 \end{lem}
 
  From the definition of the switching function we have, $\phi =(\phi_1(t),\phi_2(t))=(\lambda_2(t)I(t), \lambda_3(t) V(t))$. At $t=T$ we have $(\phi_1(T),\phi_2(T))=(0,0)$ because $(I(T),V(T))=(0,0)$. Since the final state has been already reached at $t=T$,  the values of the optimal controls  are not significant at this time. From equation $(3.11)$ we see that at $t=T$ we have,
  $$\dot V(T)= - \epsilon_1 \mu_{21}(T-\tau_1) I(T-\tau_1)= -y$$
   With this $\dot V(T)$,  the condition  $0<\lambda_3(T) < \frac{1}{y}$ of the above lemma 3.1 can be re-written as $$0 < \mu_{21}(T-\tau_1) < \frac{1}{\epsilon_1 \lambda_3(T) I(T-\tau_1)}$$
  
  Therefore,  $\textbf{u}_\text{opt}(T)= \textbf{u}_\text{max}$ if the control $\mu_{21}(T-\tau_1)$ is chosen such that $$\mu_{21}(T-\tau_1) < \frac{1}{\epsilon_1 \lambda_3(T) I(T-\tau_1)}$$.

  Now using the above lemma we prove the following result which gives the stronger property of the optimal strategy.
  
  \begin{thm}
  Let lemma 3.1 hold. If $\mu_{12}^{\text{max}} > \frac{\lambda_3(T)-\lambda_2(T)(\mu+x)}{\alpha \lambda_2(T)}$ then $\textbf{u}_\text{opt}(T)= \textbf{u}_ \text{max}$ for all $t$ in $[0, T]$
  \begin{proof}
  Let us assume that lemma 3.1 holds and $\mu_{12}^{\text{max}} > \frac{\lambda_3(T)-\lambda_2(T)(\mu+x)}{\alpha \lambda_2(T)}$
  
  Now from the canonical equations $(3.12)$ we have,
 \begin{align*}
\frac{d\lambda_1}{dt} \bigg|_{t=T} &= \lambda_1(T) \mu  \\\\
\frac{d\lambda_2(t)}{dt} \bigg|_{t=T} &= \lambda_2(T)(\mu + x + \alpha \mu_{12}(T)) - \lambda_3(T) b \\ \\
\frac{d\lambda_3(t)}{dt} \bigg|_{t=T} &= \lambda_1(T)\beta S(T) + \lambda_3(T)(\alpha \mu_{22}(T)+ y+ \mu)
\end{align*}
  Now since $\lambda_1(T)$ is arbitrary from lemma 3.1. Let   $\lambda_1(T) > 0$. We also know that $ \lambda_3(T) > 0$, therefore,
  $$
  \frac{d\lambda_3(t)}{dt} \bigg|_{t=T} > 0$$
  Using the continuity of $\lambda_3(t)$, there exists $\epsilon > 0$ such that $\lambda_3(t) > 0$ for $t \in [T-\epsilon, T]$. As long as $\lambda_1(T) > 0$, we have $0 < \lambda_3(t) <\lambda_3(T) $ for $t \in [s, T]$. Similarly since $\mu_{12}^{\text{max}} > \frac{\lambda_3(T)-\lambda_2(T)(\mu+x)}{\alpha \lambda_2(T)}$ we have,
  $$
  \frac{d\lambda_2(t)}{dt} \bigg|_{t=T} > 0$$
  Using the continuity of $\lambda_2(t)$,  we have $0 < \lambda_2(t) <\lambda_2(T) $ for $t \in [s, T]$. Therefore, in the interval $[s,T]$ we have $\textbf{u}_\text{opt}(T)= \textbf{u}_\text{max}$.
  
  Now we will show that $s=0$. Let us assume that $s \neq 0$ and $\lambda_1(s)=0$ for $0<s<T$. Since $\lambda_1(T) > 0$, using the continuity of $\lambda_1(t)$ over $[0, T]$ we get,
   $$
  \frac{d\lambda_1(t)}{dt} \bigg|_{t=s} > 0$$
  
  But from the canonical equation (3.12) we have at $t=s$
  $$
  \frac{d\lambda_1(t)}{dt} \bigg|_{t=s} = -\chi_{[0, T-\tau]} \lambda_2(s+\tau_1) \beta V(s)$$
  Since $\lambda_2(t) > 0$ and $V(t)>0$ on $[s, T]$ we get
 $$ \frac{d\lambda_1(t)}{dt} \bigg|_{t=s} < 0$$
   This leads to a contradiction. Hence $s=0$ and $\textbf{u}_\text{opt}(T)= \textbf{u}_\text{max}$ on [0, T].
  \end{proof}
  \end{thm}

  \subsubsection{\textbf{Numerical Simulation of Time Optimal Control Problem}}
  In this section we  numerically illustrate the theory discussed in the previous sections and this is done using MATLAB software. The figures in this section display the optimal trajectory of the state variables, switching functions, co-state variables and the optimal control functions. The initial and the terminal state are fixed and using parameter values from table 4 and the terminal state of the system $(3.9)-(3.11)$ we obtain the optimal value of the control at the terminal state  $u^*(T)=(\mu_{11}^*(T),\mu_{12}^*(T),\mu_{21}^*(T),\mu_{22}^*(T))$. Using this we fix the range of the control variable  $[0, u_{max}]$. Then with the objective of reaching the terminal state from the given initial state, we obtain the initial values of the co-state variables with various combination of trial-error guesses by ensuring that the Hamiltonian function $H=0$ along the optimal trajectory. The Runge-kutta method of $4^{\text {th}}$ order is used to simulate the system $(3.9)-(3.11)$ and the canonical equation $(3.12)$. The optimal control value  switches between $0$ and maximum value depending on the sign of the switching functions given in $(3.17)-(3.20)$. The time to reach the terminal state is calculated for each of the cases considered. The Hamiltonian function is monitored throughout the process. The step size for the simulation is taken to be $h=0.01$ and based on this the time to reach the desired infection free equilibrium state is re-scaled by multiplying $T \times 10^{-2}$. \\

  In  figure (7, 8) we plot the optimal trajectory of the state variables, switching functions, co-state variables and the optimal control functions over time. The figure depicts the possibility of driving the system $(3.9)-(3.11)$ from given initial state $(1000, 80, 60)$ to the desired infection free equilibrium state $(20 , 0,0)$. The minimum values  of the controls $(\mu_{11}^*, \mu_{12}^* ,\mu_{21}^*, \mu_{22}^*, )$ are taken as zero and the maximum values are taken as $2$ and the values of other fixed parameters for the simulation is taken from table 4.  The final time taken by the system $(3.9)-(3.11)$ to reach the desired terminal state is calculated to be  $T=1947\times 10^{-2} = 19.47$ units of time. This example  is a case with  one switch. It is observed that  optimal control functions $\mu_{11}^*$ and $\mu_{12}^*$ switches its values from maximum (i.e. $2$) to minimum (i.e. $0$) whereas,  the optimal control functions $\mu_{21}^*$ and $\mu_{22}^*$ switches its values from  minimum to maximum. The nature of these switches in the values of optimal controls depends on the sign of the switch function $\phi_1$ and $\phi_2$ as discussed in theorem 3.4. This example is applicable in the context of eliminating the virus particle from the body and we see from the figure that the infected cell increases initially  for a certain period of time owing to the presence of virus particles and then decreases and comes down to zero as viral load decreases with the administration of higher quantity of the controls.  
  
 \newpage	
\begin{figure}[hbt!]
\begin{center}
\includegraphics[width=3.1in, height=2.5in, angle=0]{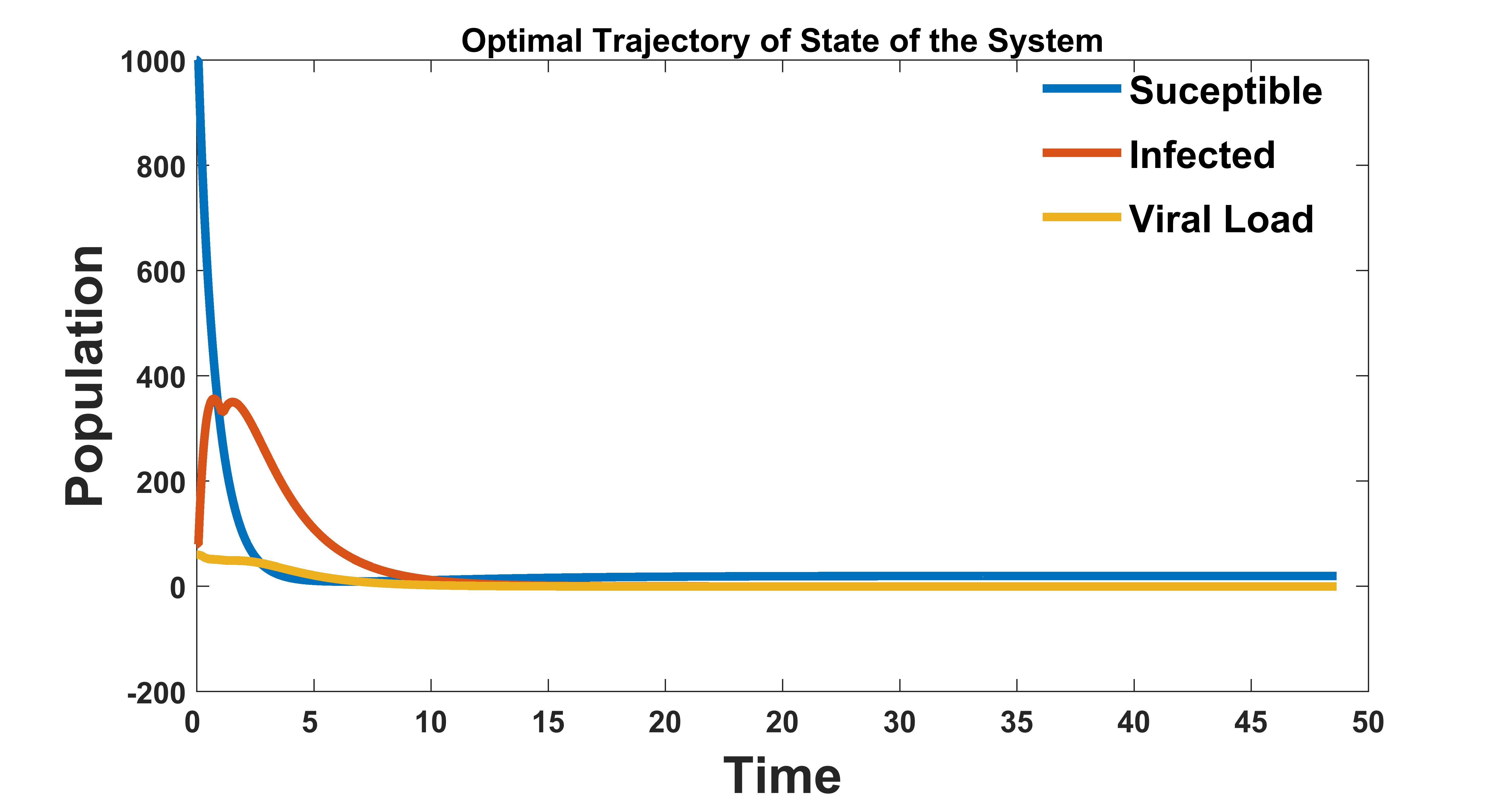}
\hspace{-.4cm}
\includegraphics[width=3.1in, height=2.5in, angle=0]{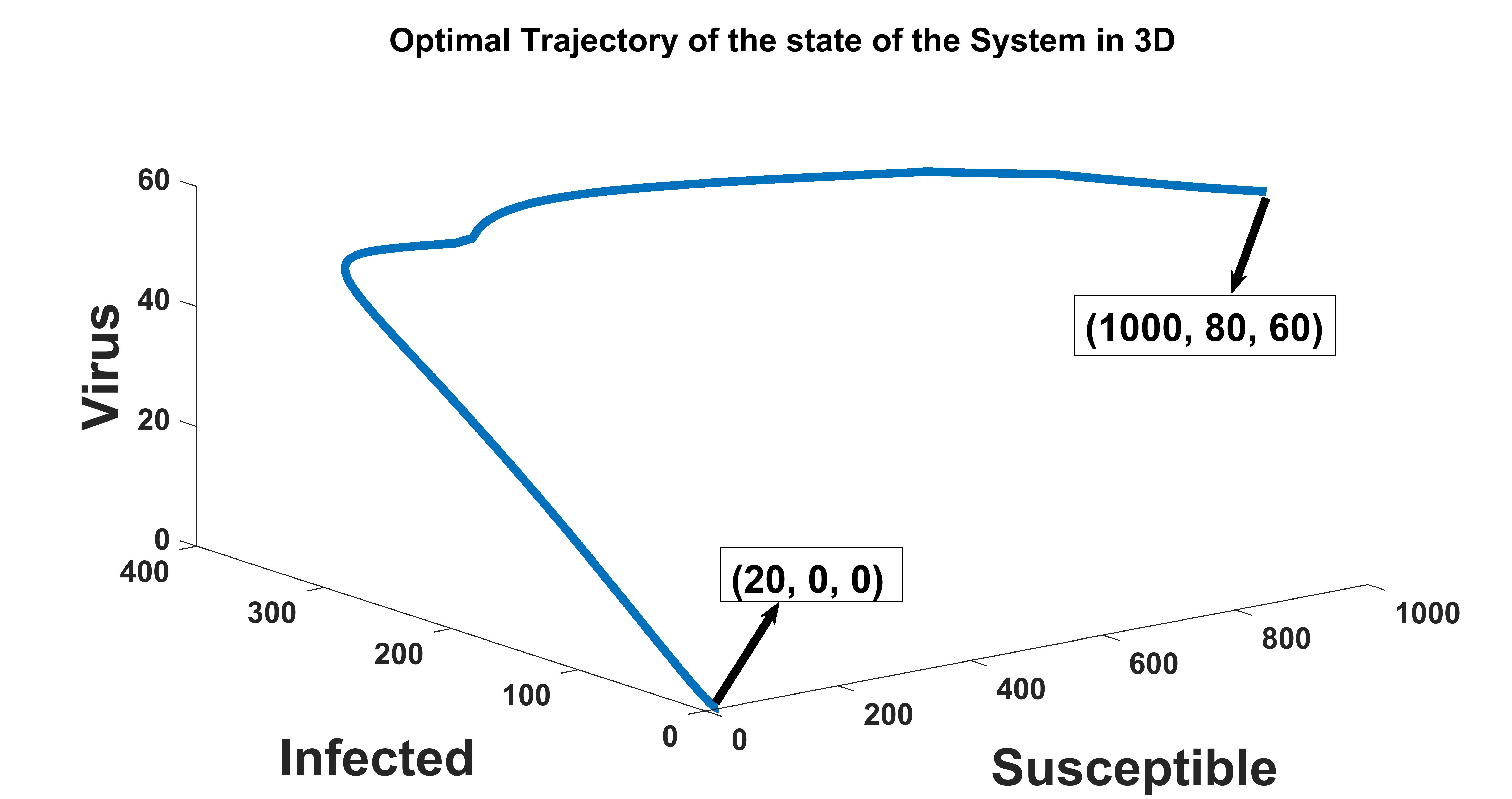}

\end{center}
\end{figure}

\vspace{1cm}

\begin{figure}[hbt!]
\begin{center}
\includegraphics[width=3.1in, height=2.5in, angle=0]{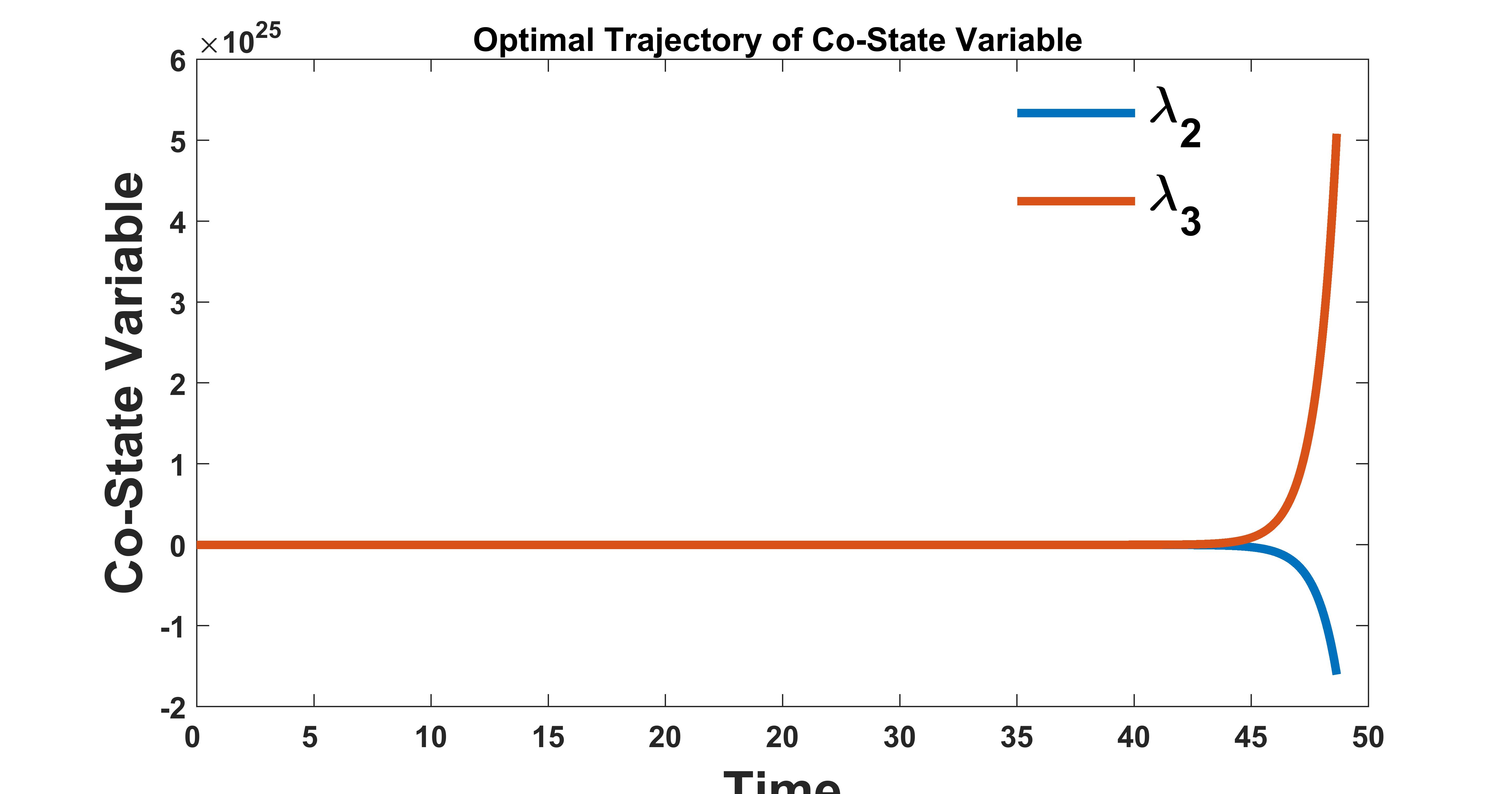}
\hspace{-.4cm}
\includegraphics[width=3.1in, height=2.5in, angle=0]{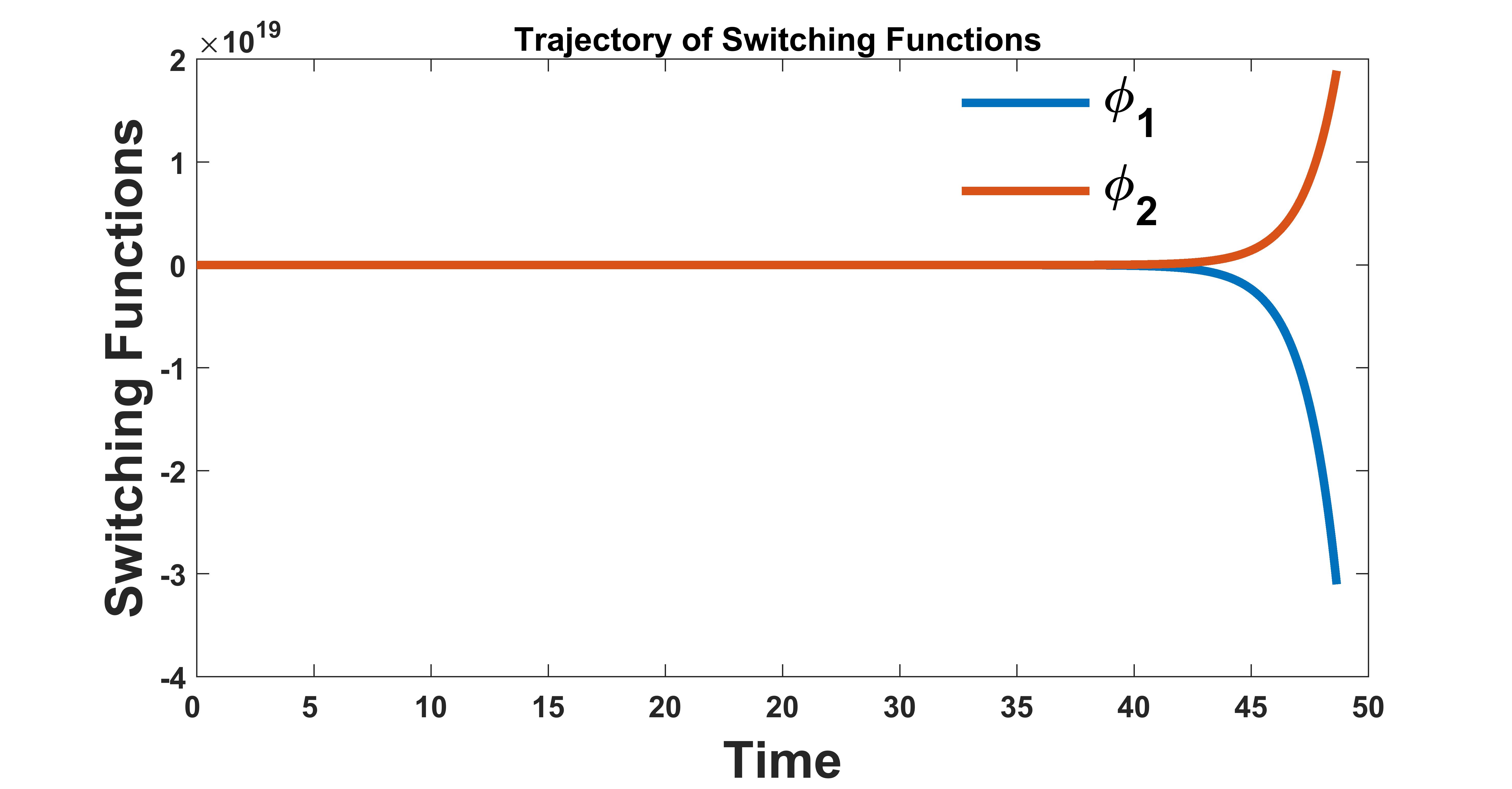}

\vspace{7mm}
\caption{Figure depicting the optimal trajectory of the time optimal control problem (3.8)-(3.11) from the initial state $(1000, 80, 60)$ to the desired terminal state $E_0 = (20, 0,0)$  with the parameter values from table 4.}
\label{1}

\end{center}
\end{figure}

\vspace{20cm}
\begin{figure}[hbt!]
\begin{center}
\includegraphics[width=3.1in, height=2.5in, angle=0]{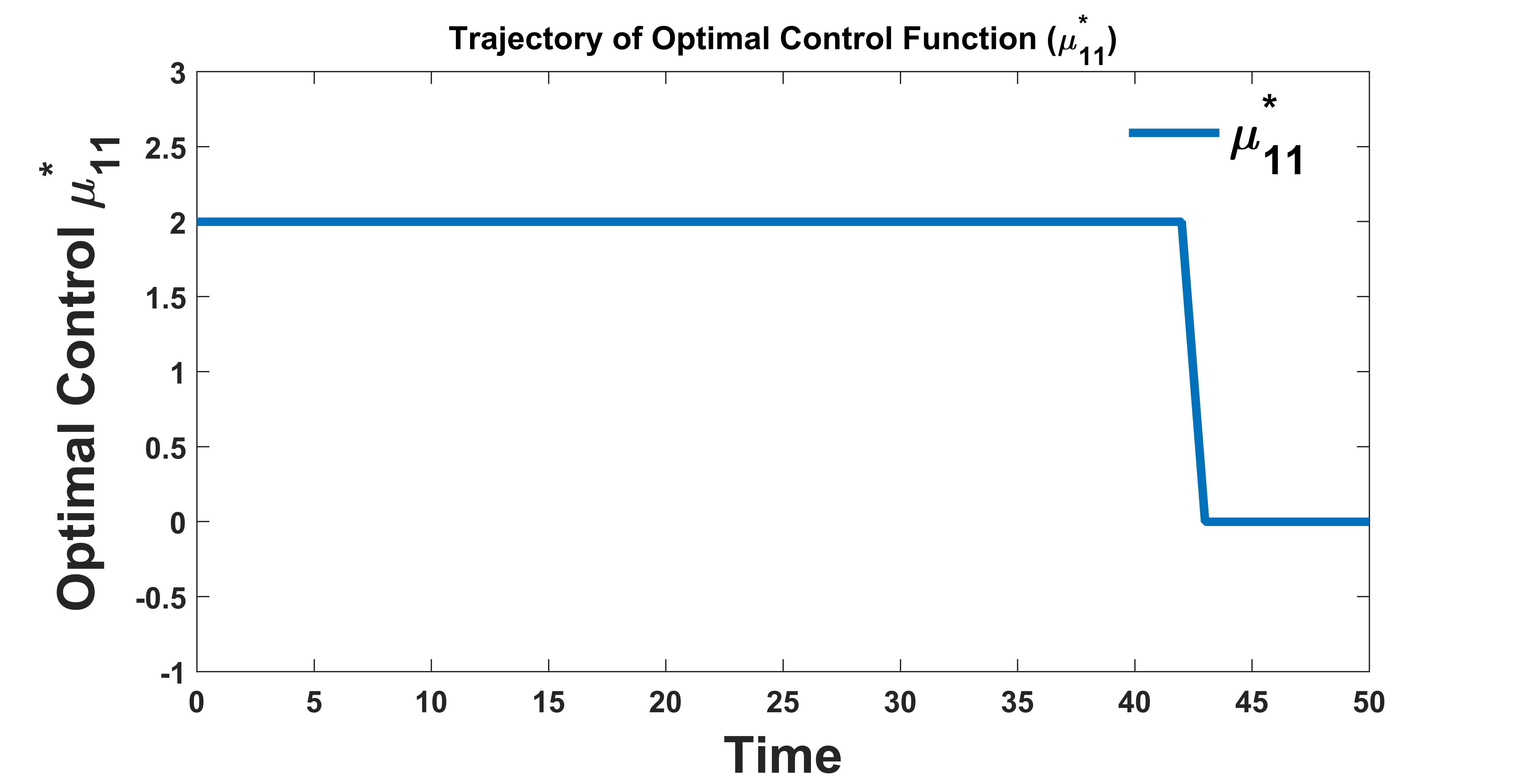}
\hspace{-.395cm}
\includegraphics[width=3.1in, height=2.5in, angle=0]{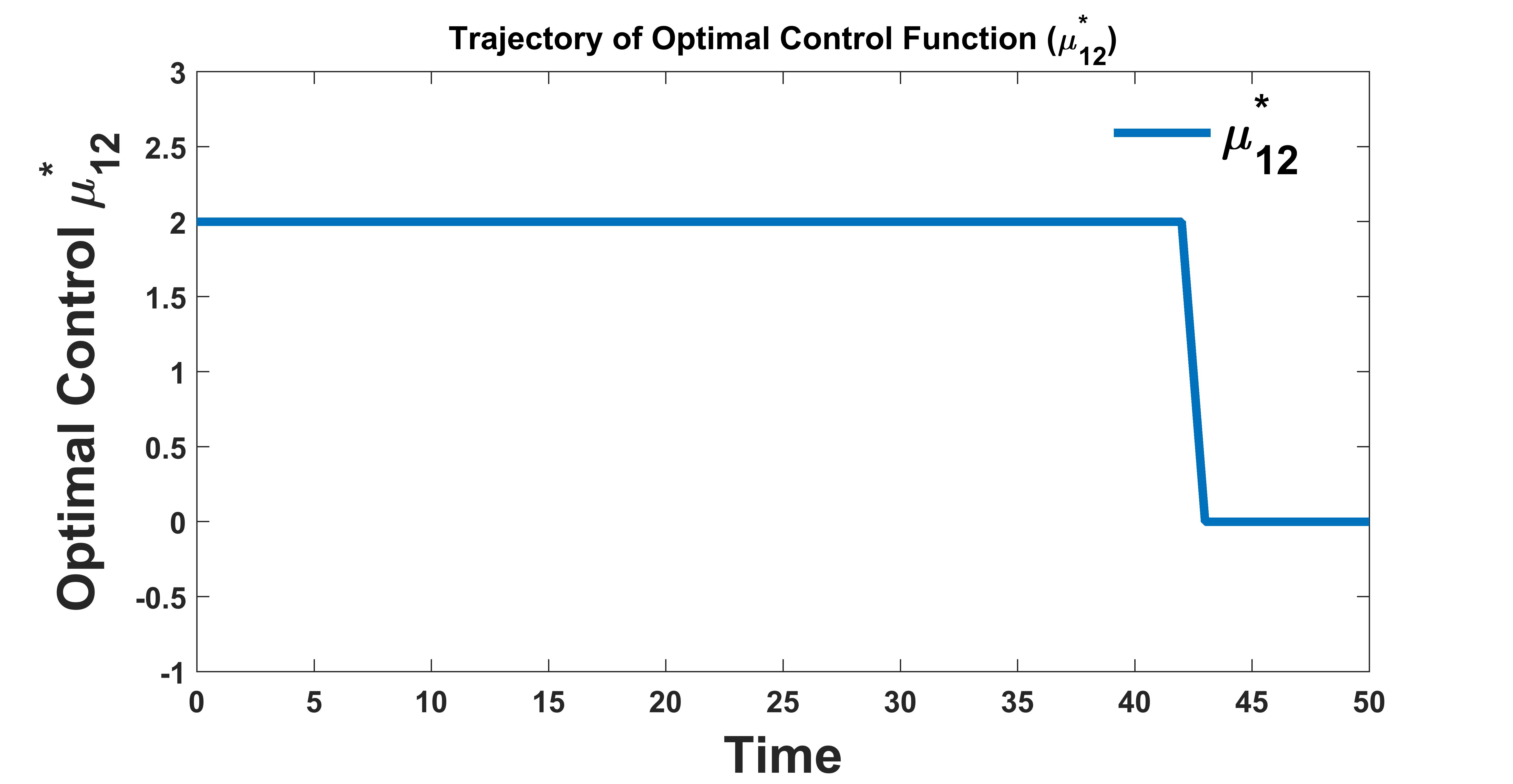}

\end{center}
\end{figure}

\vspace{-3mm}

\begin{figure}[hbt!]
\begin{center}
\includegraphics[width=3.1in, height=2.5in, angle=0]{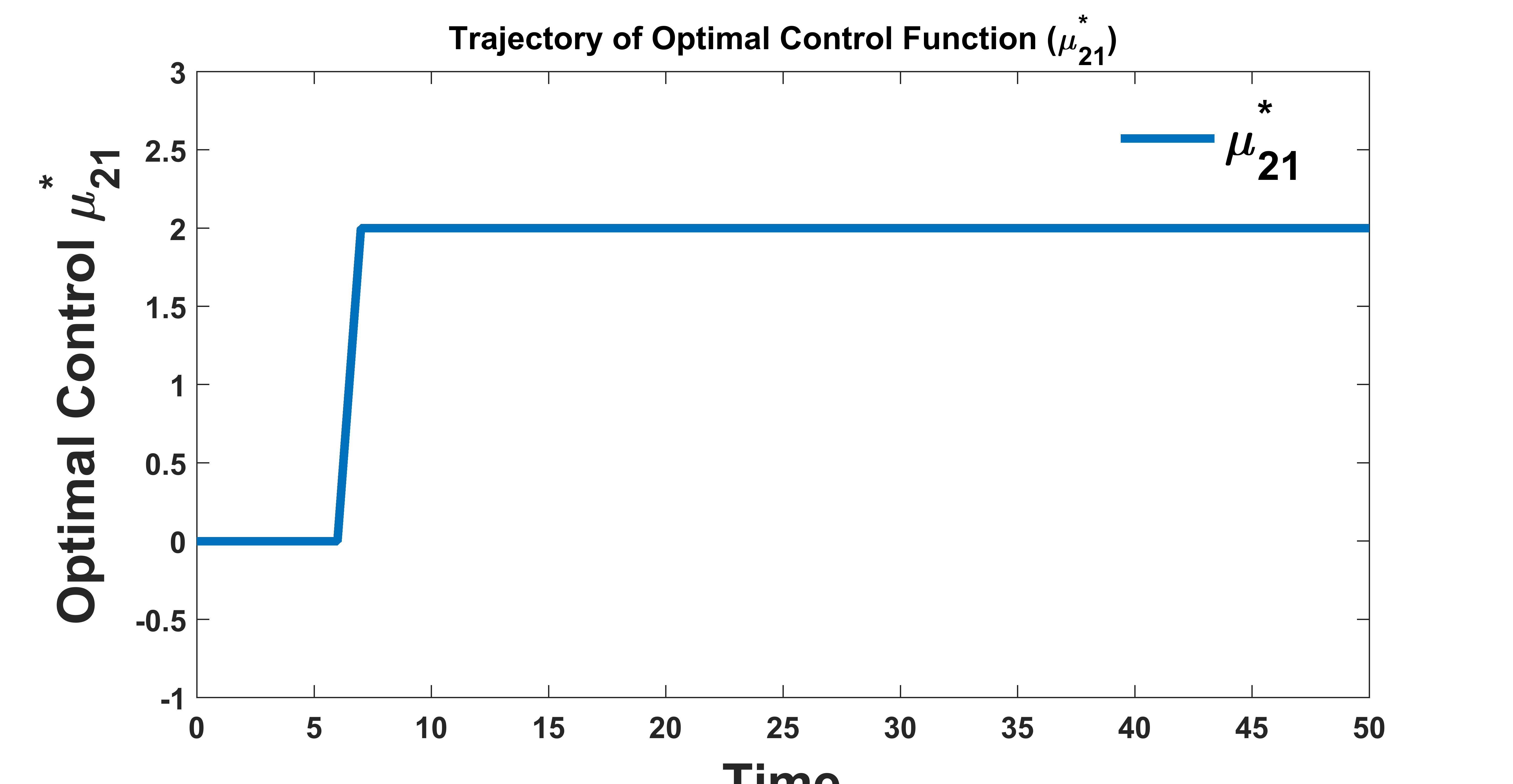}
\hspace{-.4cm}
\includegraphics[width=3.1in, height=2.5in, angle=0]{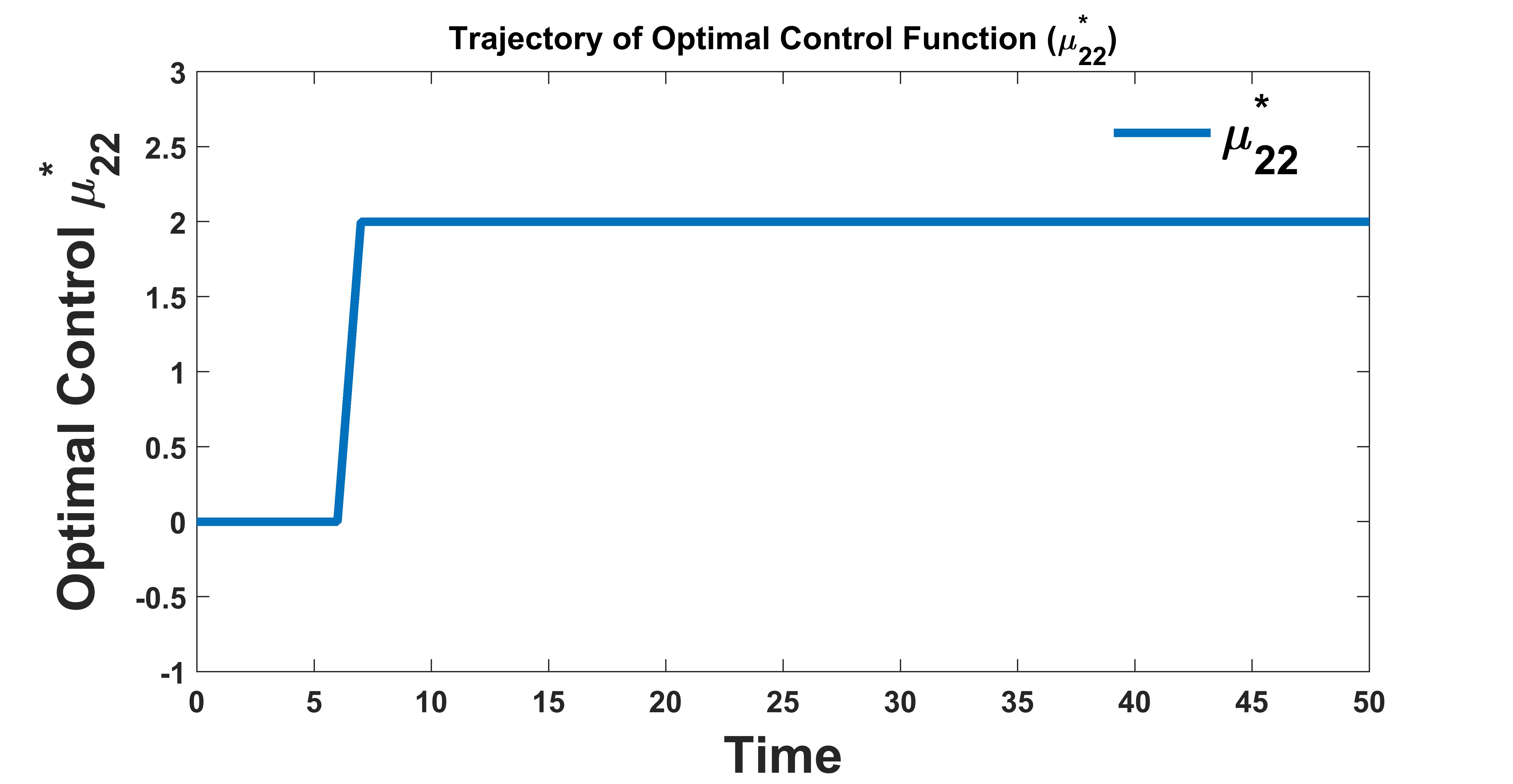}

\vspace{7mm}
\caption{Figure depicting the  trajectory of optimal control functions of the time optimal control problem (3.8)-(3.11).}
\label{1}
\end{center}
\end{figure}

In  figure (9, 10) we plot the optimal trajectory of the state variables, switching functions, co-state variables and the optimal control functions over time. The figure depicts the possibility of driving the system $(3.9)-(3.11)$ from given initial state $(1000, 80, 60)$ to the desired infection free equilibrium state $(20 , 0,0)$. The minimum values  of the controls $(\mu_{11}^*, \mu_{12}^* ,\mu_{21}^*, \mu_{22}^*, )$ are taken as zero, same as in the previous example but  the maximum values  are now  taken to be $2.5$ and the values of other fixed parameters are from table 4. This example is a case without any switch in the values of optimal controls. This case is an illustration of theorem 3.5 where, $\lambda_2(t) > 0$ and $\lambda_3(t) > 0$ $\forall t \in [0, T]$ and $\mu_{11}^*(t)= \mu_{\text{max}}$, $\mu_{21}^*(t)= \mu_{\text{max}}$, $\mu_{12}^*(t)= \mu_{\text{max}}$ and $\mu_{22}^*(t)= \mu_{\text{max}}$ $\forall t \in [0, T]$. In this case the time to reach the desired infection free equilibrium state is calculated to be $T=1606\times 10^{-2} = 16.06$ units of time which is lesser than the previous example. This example illustrates that the infection free equilibrium state could be achieved by the system in much lesser time maintaining the administration of the controls at maximum levels throughout  the observation period.

\newpage
\begin{figure}[hbt!]
\begin{center}
\includegraphics[width=3.1in, height=2.6in, angle=0]{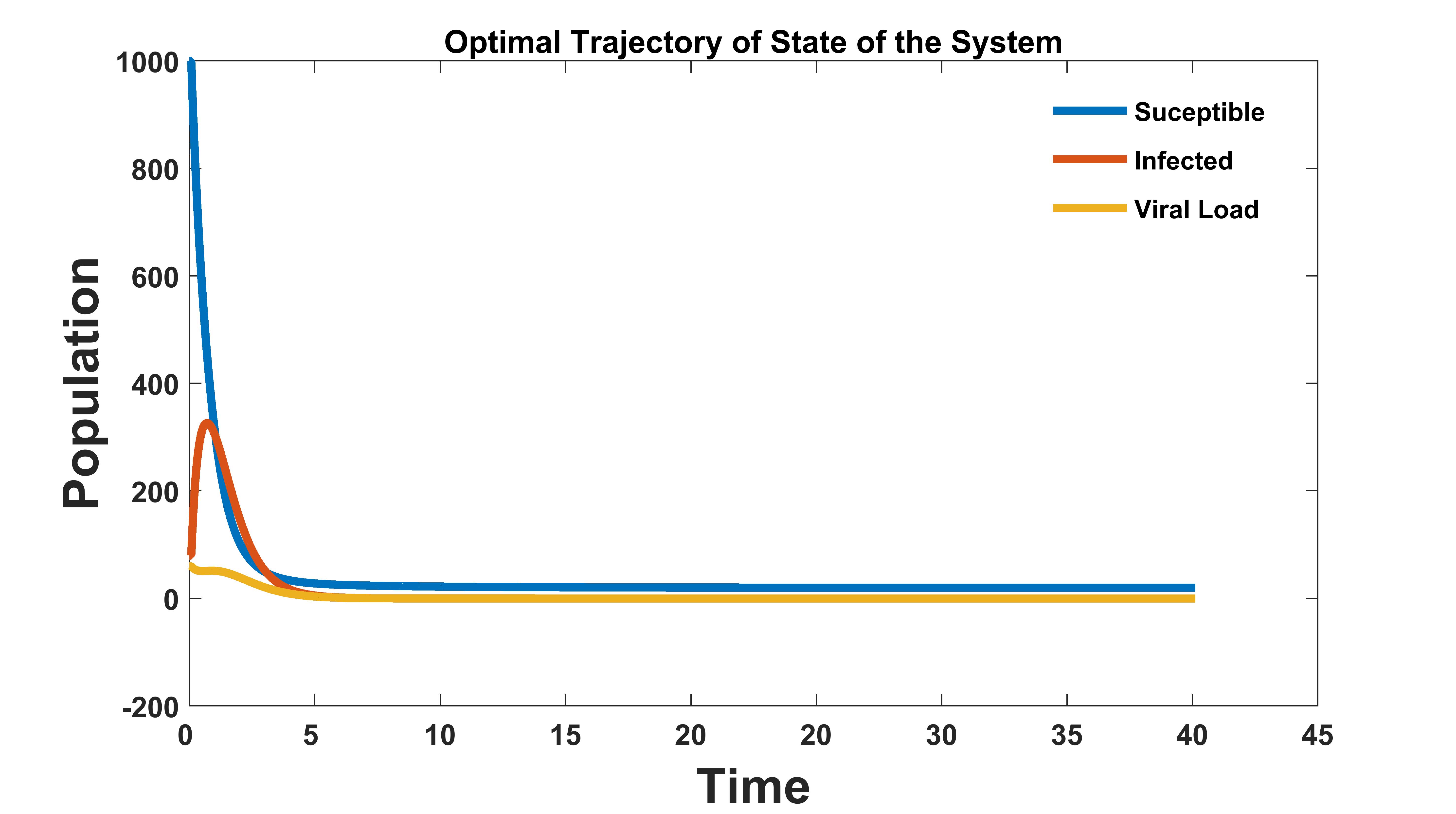}
\hspace{-.4cm}
\includegraphics[width=3.1in, height=2.6in, angle=0]{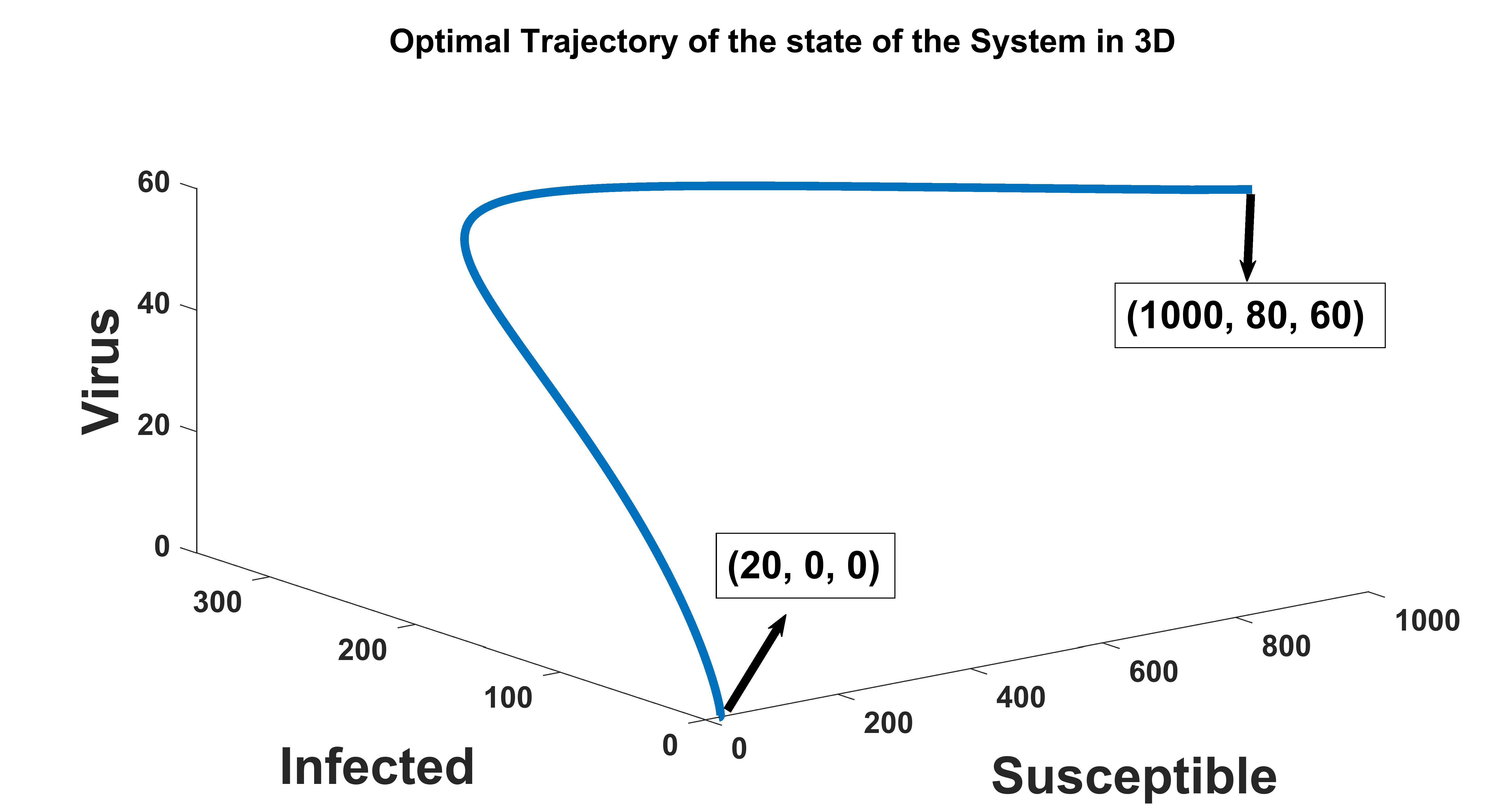}

\end{center}
\end{figure}

\vspace{1cm}

\begin{figure}[hbt!]
\begin{center}
\includegraphics[width=3.1in, height=2.6in, angle=0]{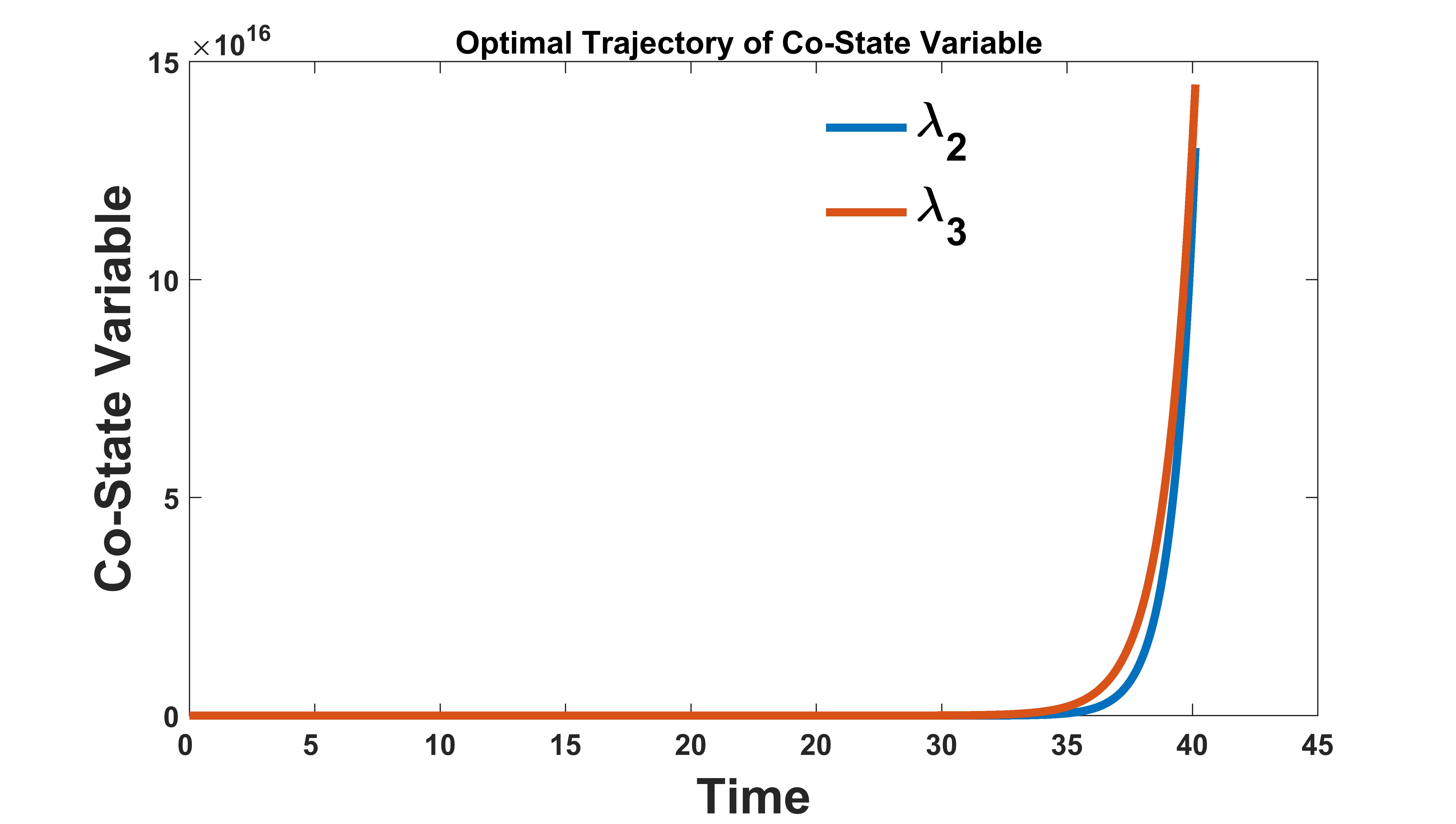}
\hspace{-.4cm}
\includegraphics[width=3.1in, height=2.6in, angle=0]{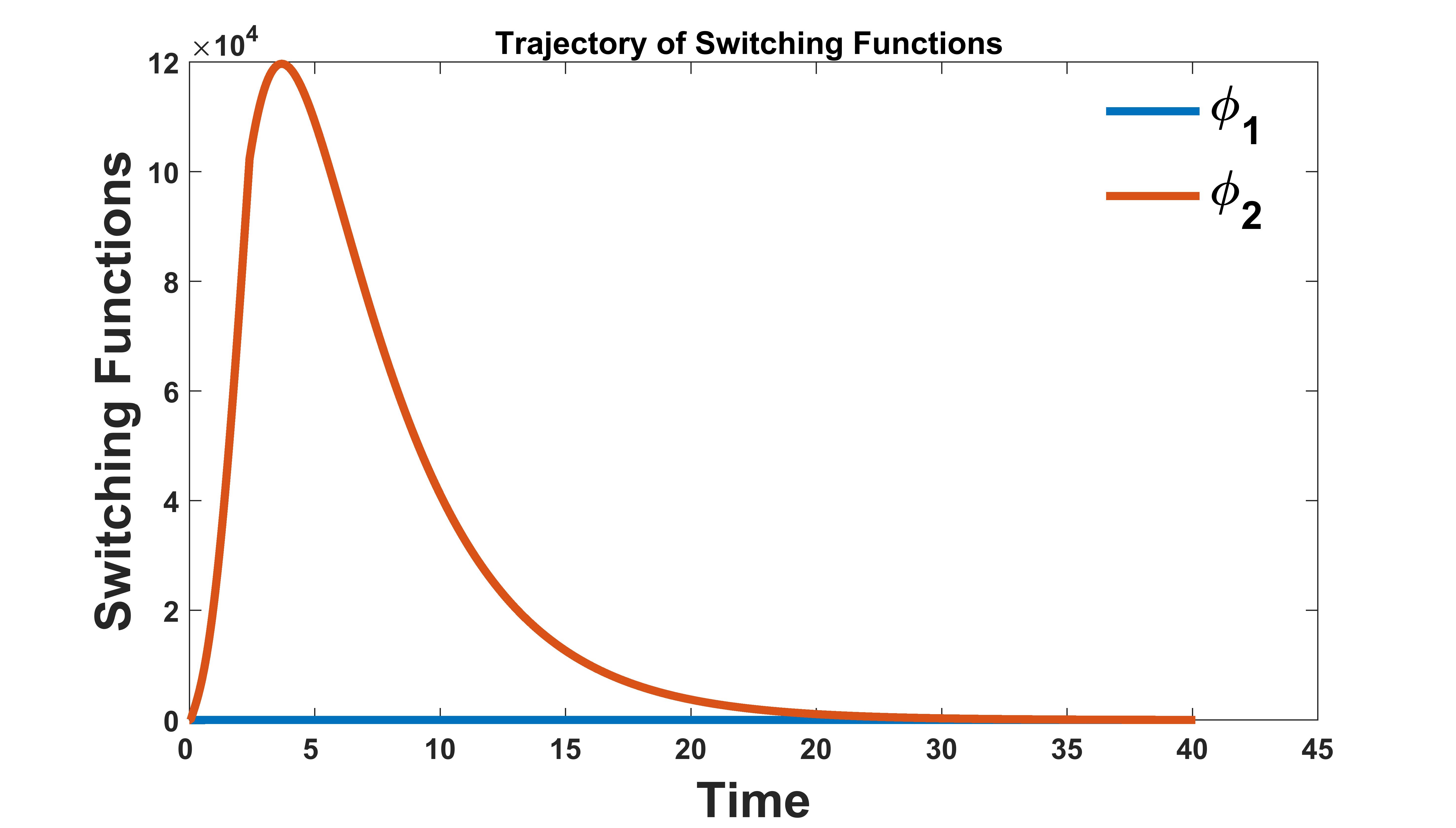}

\vspace{7mm}
\caption{Figure depicting the optimal trajectory of the time optimal control problem $(3.8)-(3.11)$ from the initial state $(1000, 80, 60)$ to the desired terminal state $E_0 = (20, 0,0)$  with the parameter values from table 4.}
\end{center}
\end{figure}

\vspace{10cm}
\begin{figure}[hbt!]
\begin{center}
\includegraphics[width=3.1in, height=2.6in, angle=0]{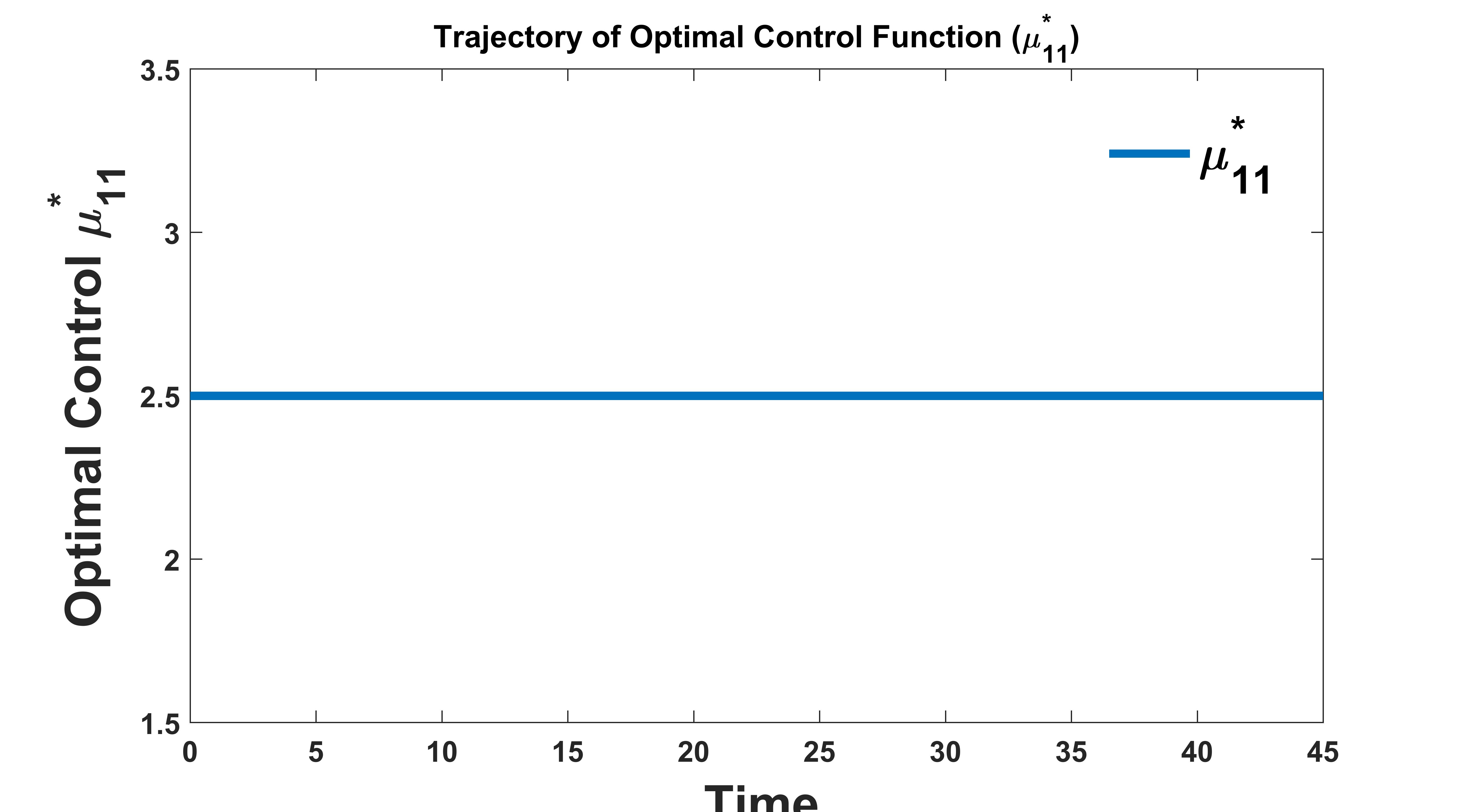}
\hspace{-.395cm}
\includegraphics[width=3.1in, height=2.6in, angle=0]{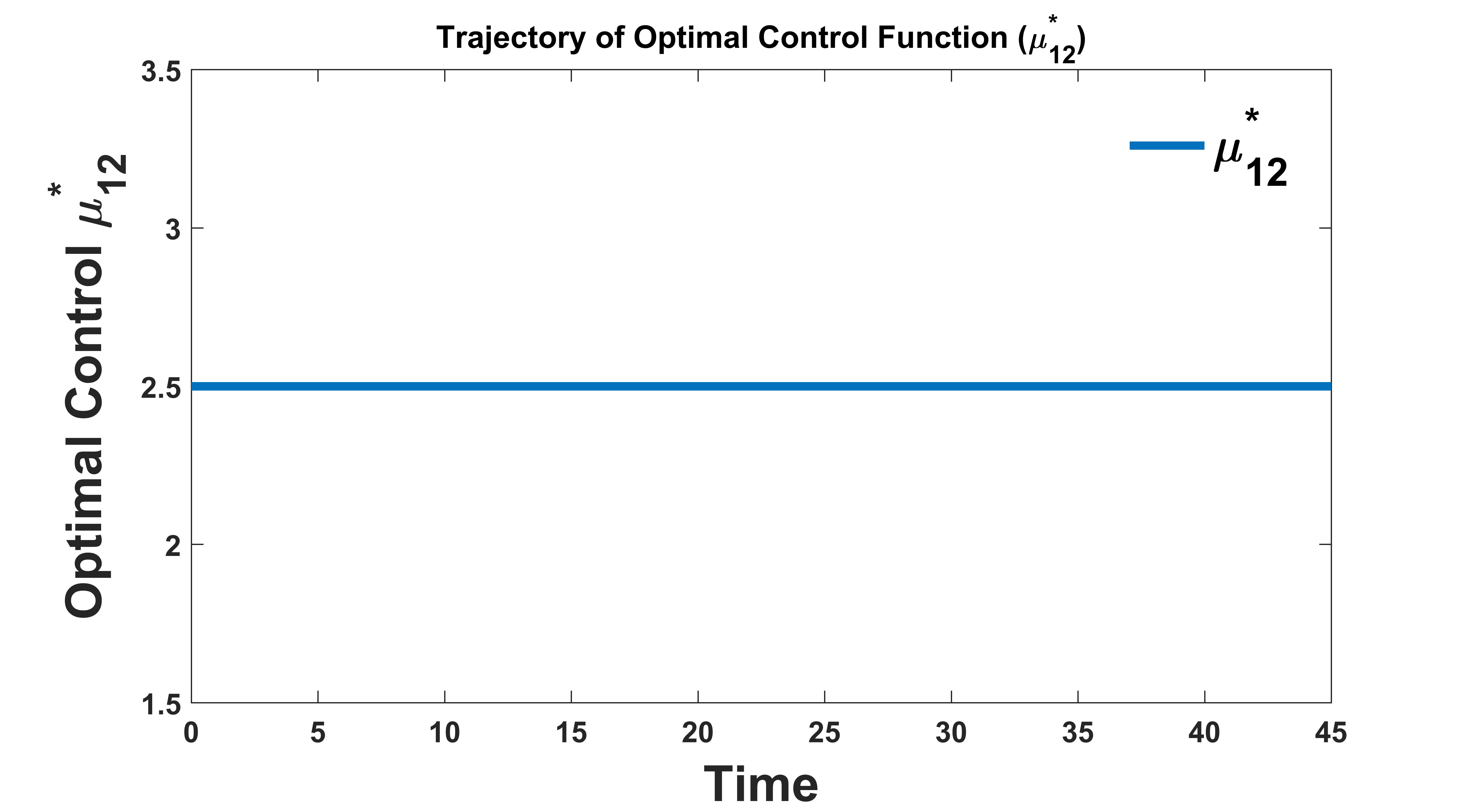}

\end{center}
\end{figure}

\vspace{-3mm}

\begin{figure}[hbt!]
\begin{center}
\includegraphics[width=3.1in, height=2.5in, angle=0]{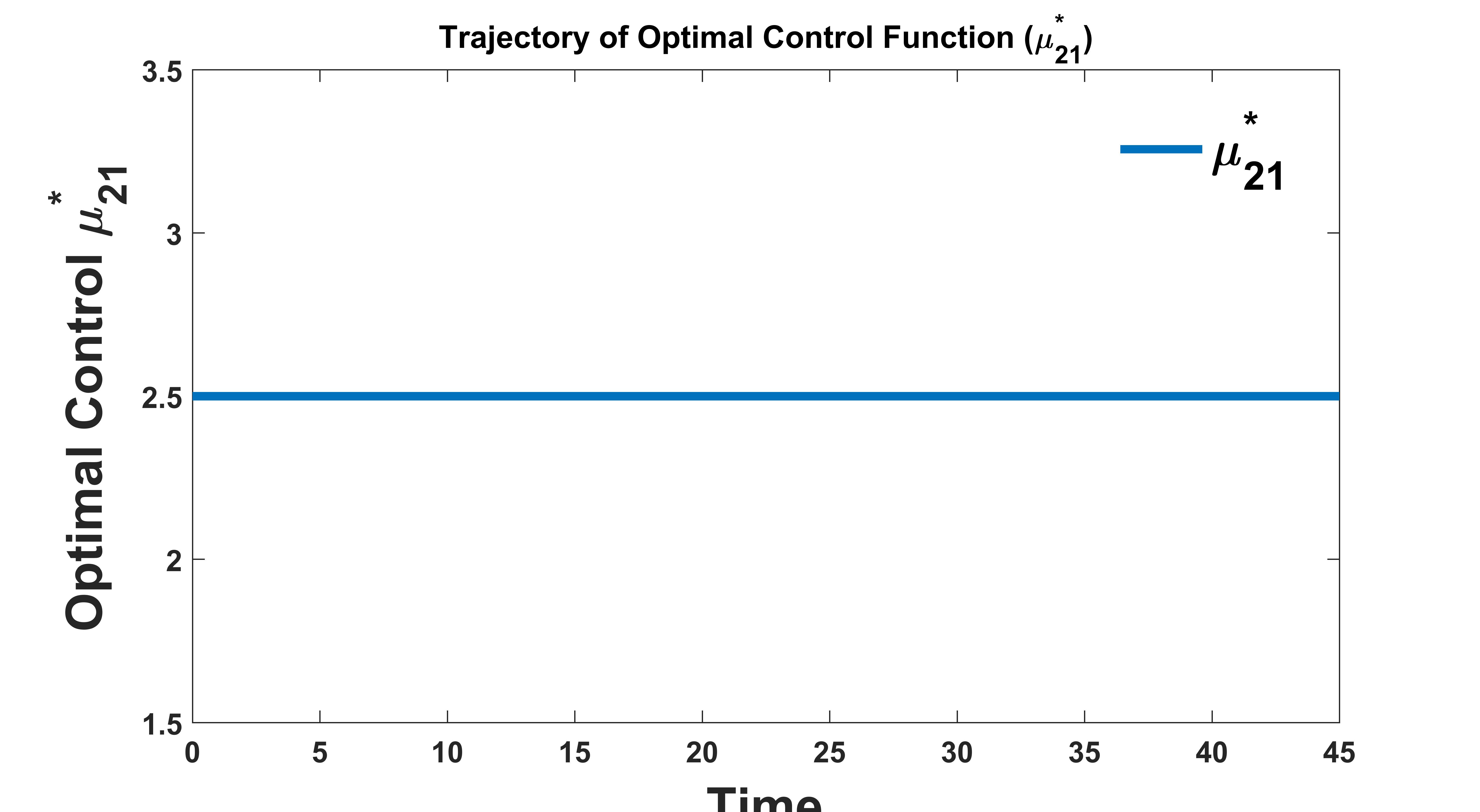}
\hspace{-.4cm}
\includegraphics[width=3.1in, height=2.5in, angle=0]{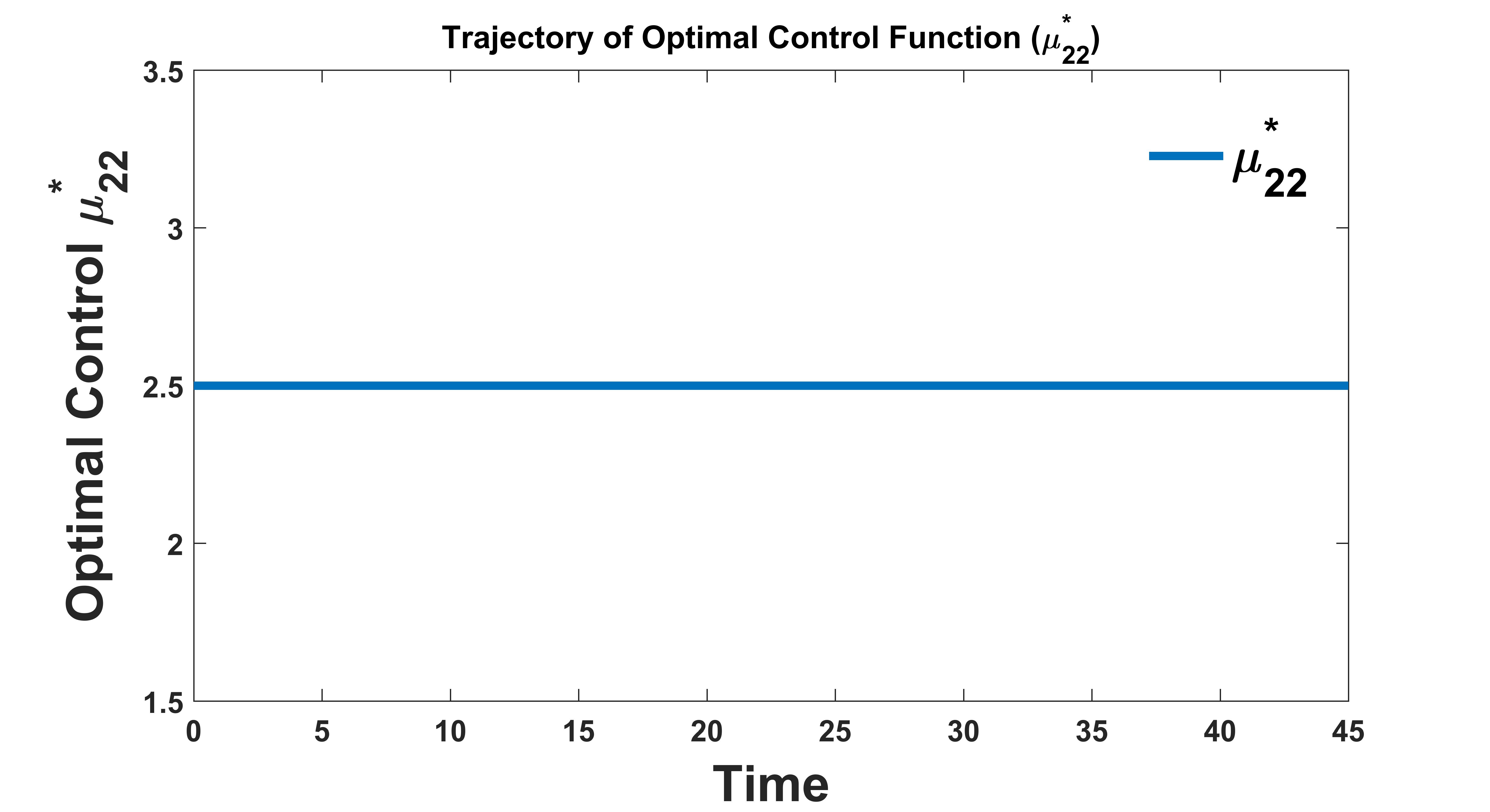}

\vspace{7mm}
\caption{Figure depicting the  trajectory of optimal control functions of the time optimal control problem (3.8)-(3.11).}
\label{1}
\end{center}
\end{figure}

\subsection{Discussions and Conclusions}
\noindent
 The outbreak of novel coronavirus in Wuhan, China marked the introduction of a virulent coronavirus into human society. It has resulted in around 154 million cases and  3.22 million deaths worldwide. With several research teams working together several mathematical as well as clinical works are being done to better understand the disease dynamics and examine the efficacies of various interventions. Although a lot of research is being done, effective approaches to treatment and epidemiological control are still lacking.

Since  COVID-19 outbreak, researchers have suggested many agents that could have potential efficacy against COVID-19.  There is no specific, effective treatment or cure for coronavirus disease 2019 (COVID-19) like  SARS-CoV and MERS-CoV. In these situation certain antiviral drugs like  chloroquine, hydroxychloroquine, lopinavir/ritonavir, and remdesivir  and certain immunomodulators such as INF and Zinc supplements are being used in the treatment. Mathematical models are known to provide useful information in short period of time and more importantly in the non-invasive way. Therefore, in this context, within-host mathematical modelling studies can be extremely helpful in understanding the natural history of this new disease, role and efficacies of the antiviral drugs (remdesivir, hydroxychloroquine etc.) and second line drugs (methylprednisolone ) individually and in combination. A SAIU compartmental mathematical model that explains the transmission dynamics of COVID-19 is developed in \cite{samui2020mathematical}.  The role of some of the control policies such as treatment, quarantine, isolation, screening, etc. are also applied to control the spread of infectious diseases \cite{kkdjou2020optimal,libotte2020determination,aronna2020model}.

In this study initially, we extended the work done in \cite{chhetri2020within} by incorporating inter-cellular time delay and studied the stability analysis of the equilibrium points. Secondly, to study the role and efficacies of the first and second line drugs an optimal control problem was framed. Fillipov existence theorem and Pontryagin Minimum Principle was used for proving the existence and obtaining the optimal solutions. Lastly, a time optimal control problem was formulated with the objective of driving the system from any given initial state to the desired infection free equilibrium state in minimal time. Using Pontryagin Minimum Principle the optimal controls were shown to be of bang-bang type with possibility of switches occurring in the optimal trajectory.

Findings from the stability analysis of the equilibrium points suggested that the infection free equilibrium point denoted by $E_0$
remained asymptotically stable for all the values of inter-cellular delay ($\tau$) provided the value of basic reproduction number $R_0$ was less than unity. As the value of $R_0$ crossed unity the infected equilibrium point $E_1$ was found to be asymptotically stable for all values of inter-cellular delay as discussed in theorem 2.3.

From the  comparative study done in the optimal control problem section 3.1, we find that when the first line antiviral drugs starts showing adverse events $(\alpha >0)$ , considering first line antiviral drugs in reduced quantity along with the second line drug could be highly effective  in reducing the infected cells and viral load in a COVID infected patients and this alternative also proved to be cost effective compared to the first line drug only case (figure (4, 5, 6)).  

In subsection 3.2 time optimal control problem was framed with the objective of driving the system from any given initial state to the desired infection free equilibrium state $E_0$ in minimum time with varying values of the first line and second line drugs. Using Pontryagin's Minimum Principle it was shown that the optimal strategy is of bang-bang time with possibility of switches between two extreme values of the optimal controls depending on the sign of switch functions. The findings from the time optimal control study indicates that with higher values of first line and second line drugs the time to reach the desired infection free state decreases (numerical illustration of theorem 3.5). This would imply that  the infected cells and viral load in the body of a  COVID infected individual becomes zero in short period of time with the higher values of the first line and second line drugs.

\quad The within-host optimal control studies to evaluate the efficacies of the anti viral drugs  without considering the adverse events of these drugs has been done in \cite{chhetri2020within} but within-host time optimal control studies in COVID-19  is first of its kind studied here and the results obtained from this can be helpful to researchers, epidemiologists, clinicians and doctors working in this field.

   \bibliographystyle{amsplain}
\bibliography{references}

\end{document}